\documentclass[12pt]{article}

\usepackage{amssymb, amsfonts, amsmath, amsthm}

\usepackage{color}
\usepackage[usenames,dvipsnames,svgnames,table]{xcolor}
\usepackage[margin = 1in]{geometry}
\usepackage{graphicx, epstopdf}
\graphicspath{{./diagrams/}}
\usepackage{caption}	
\usepackage{subcaption}
\usepackage{multirow}
\usepackage{url,  mathrsfs}
\usepackage{hyperref}

\usepackage{marginnote}

\newcommand{\abs}{\textnormal{abs}}

\usepackage[noend]{algorithm2e}
\makeatletter
\newcommand{\nosemic}{\renewcommand{\@endalgocfline}{\relax}}
\newcommand{\dosemic}{\renewcommand{\@endalgocfline}{\algocf@endline}}
\newcommand{\pushline}{\Indp}
\newcommand{\popline}{\Indm\dosemic}
\makeatother

\newcommand{\R}{\mathbb{R}}
\newcommand{\C}{\mathbb{C}}
\newcommand{\Q}{\mathbb{Q}}
\newcommand{\N}{\mathbb{N}}
\newcommand{\Z}{\mathbb{Z}}
\renewcommand{\P}{\mathbb{P}}
\newcommand{\1}{{\bf 1}}

\newcommand\V{\mathcal{V}}
\newcommand\ps{\{\!\{t\}\!\}}

\DeclareMathOperator\val{val}
\DeclareMathOperator\coeff{coeff}
\DeclareMathOperator\In{in}
\DeclareMathOperator\Trop{Trop}
\DeclareMathOperator\TropC{Trop_\C}
\DeclareMathOperator\TropR{Trop_\R}

\usepackage{bbm}
\renewcommand\k{\mathbbm{k}}

\theoremstyle{plain}
\newtheorem{Thm}{Theorem}
\newtheorem{Prop}[Thm]{Proposition}

\newtheorem{Lemma}[Thm]{Lemma}

\newtheorem*{Thm*}{Theorem}
\newtheorem*{Prop*}{Proposition}
\newtheorem*{Cor*}{Corollary}
\newtheorem*{Lemma*}{Lemma}
\newtheorem*{Conjecture*}{Conjecture}

\theoremstyle{definition}

\newtheorem{Def}[Thm]{Definition}
\newtheorem{Example}[Thm]{Example}
\newtheorem{Remark}[Thm]{Remark}

\newtheorem*{Constr*}{Construction}
\newtheorem*{Def*}{Definition}
\newtheorem*{Example*}{Example}
\newtheorem*{Remark*}{Remark}

\DeclareMathOperator{\sign}{sign}

\usepackage{etoolbox} 
\newtoggle{clearpages} 
\toggletrue{clearpages}  

\usepackage{placeins}
\newtoggle{barriers} 
\toggletrue{barriers}

\title{Computing complex and real tropical curves\\ using monodromy}
\author{Daniel A. Brake\thanks{Department of Applied
and Computational Mathematics and Statistics, 
University of Notre Dame, Notre Dame, IN 46556 ({\tt dbrake@nd.edu}, \url{danielthebrake.org}).}
\and 
Jonathan D. Hauenstein\thanks{Department of Applied
and Computational Mathematics and Statistics, 
University of Notre Dame, Notre Dame, IN 46556 ({\tt hauenstein@nd.edu}, \url{www.nd.edu/\~jhauenst}).}
\and
Cynthia Vinzant\thanks{Department of Mathematics, 
North Carolina State University, Raleigh, NC 27695 ({\tt clvinzan@ncsu.edu}, \url{www.math.ncsu.edu/\~clvinzan}).}}

\begin{document}

\maketitle

\begin{abstract}

\noindent Tropical varieties capture combinatorial information about how coordinates 
of points in a classical variety approach zero or infinity.  We 
present algorithms for computing the rays of a complex and real 
tropical curve defined by polynomials with constant coefficients.
These algorithms rely on homotopy continuation, monodromy loops, and Cauchy integrals.
Several examples are presented which are computed
using an implementation that builds on the 
numerical algebraic geometry software {\tt Bertini}.
\end{abstract}

\section*{Introduction}

Tropical geometry is a field of mathematics that uses combinatorial structures to study problems in algebraic geometry. 
It has proven to be a powerful tool for understanding real and complex varieties. 
Computations with tropical varieties can extract delicate combinatorial properties of algebraic varieties, 
and tropical methods can be used to construct finely-tuned examples of varieties with desired properties, e.g., \cite{TropCentralPath, Viro08}.

Most of the current algorithms for computing tropical varieties 
are symbolic and are restricted to tropical varieties 
defined over algebraically-closed fields. 
The past few years have seen the development of new techniques 
based on numerical algebraic geometry
for computing complex tropical varieties in the case of 
hypersurfaces \cite{HauensteinSottile12} and curves \cite{jensen2014computing}.  
In this paper, we present a new algorithm for computing complex tropical curves, 
which we then enhance to compute real tropical curves.
We have implemented these algorithms 
using a combination of Matlab and {\tt Bertini} \cite{BHSW06}
available at \cite{BertiniTropical}.

We anticipate that the numerical 
computation of tropical curves represents a first step towards the 
numerical computations 
of higher-dimensional tropical varieties. 
Another reason to focus on curves, as mentioned in \cite{jensen2014computing}, 
is that the computation of tropical curves is an internal step in the 
computation of other tropical varieties. 
Moreover, to the best of our knowledge, 
no other computational techniques exist for real tropical curves.
We hope to extend this to higher-dimensional real tropical varieties as well. 

The remainder of the paper is organized as follows. 
In the first section, we summarize the necessary background and 
notation for tropical varieties.
Cauchy's integral formula and projectivization are 
discussed in Section~\ref{sec:Cauchy}. 
In Section~\ref{sec:Algorithms}, we present our algorithms for computing 
complex and real tropical curves and prove their correctness. 
Implementation of these algorithms is discussed in Section~\ref{sec:implementation}.  Finally, we conclude in Section~\ref{sec:examples}
with examples.

\section{Background on tropical geometry}\label{sec:trop}

Let $\k$ be a field and $I$ be an ideal in the polynomial ring $\k[x_1, \hdots, x_n]$. For any subset $S$ of the field $\k$, we denote $\V_S(I)$ as zero set in $S^n$ of the 
polynomials in $I$. 
The most often cases are $S=\k$ and the set of non-zero elements $S=\k^*$.  

The set of Puiseux series over a field $\k$, denoted $\k\ps$, is the union over all positive integers $n$ of 
the formal Laurent series in $t^{1/n}$, namely
\begin{equation}
\k\ps \;\; =\;\; \bigcup_{n\geq 1} \k(\!(t^{1/n})\!) \;\;= \;\; 
\bigcup_{n\geq1} \left.\left\{\sum_{j=k}^{\infty}c_{j}t^{j/n}
~\right|
k\in\mathbb{Z}, \;\;c_{j}\in \k \right\}.
\label{eqn:puiseux}
\end{equation}
The field $\C\ps$ is algebraically closed and $\R\ps$ is real closed. The field of Puiseux series has a valuation,
$\val :\k\ps^* \rightarrow \mathbb{Q}$, and coefficient map, $\coeff:\k\ps^* \rightarrow \k^*$, given by 
\[
\val \left(\sum_qc_qt^q\right) \;= \; \min\{q \; | \;c_{q}\neq 0\}
\ \quad \text{ and } \quad \ 
\coeff(y) \;= \; 
c_{\val(y)}.
\] 
If a Puiseux series $y$ converges for $t$ in a neighborhood of zero then $\coeff(y)t^{\val(y)}$ 
gives its order of growth around $t=0$.  Both maps $\val(\cdot)$ and $\coeff(\cdot)$ extend to $\k\ps^n$ coordinate-wise. 

Given an ideal $I \subset \k[x_1, \hdots, x_n]$, consider 
the zeros of the polynomials in $I$ over $\k\ps^n$ and their valuations  in $\Q^n$. For technical reasons, we 
actually consider the negative of the valuation map, corresponding to the ``max'' convention.  Letting this run over all solutions and taking the Euclidean 
closure in $\R^n$ yields the \emph{tropical variety} over $\k$ of the ideal $I$. 

\begin{Def} 
Given an ideal $I\subset \k[x_1, \hdots, x_n]$, the {\bf tropical variety over $\k$},  denoted $\Trop_\k(I) \subset \R^n$, is the closure of the image of its variety over $\k\ps^*$ under the negative valuation map, 
\[
\Trop_\k(I)\;\;=\;\; -\overline{\val(\V_{\k\ps^*}(I))} .
\]
We call $\TropC(I)$ the {\bf tropical variety} of $I$ and $\TropR(I)$ the {\bf real tropical variety} of $I$.  
As $\Trop_\k(I)$ only depends on the variety $\V_{\k^*}(I)$, we will also use the notation $\Trop_\k(\V(I))$. 
\end{Def}

In this paper, we will be particularly interested in locally convergent Puiseux series.~For 
$\k = \R$, $\C$, replacing $\k\ps$ with locally convergent Puiseux 
series $\k\ps_{conv}$ does not~change the image of a variety under the valuation map.  That is, 
\mbox{$\Trop_\k(I) =-\overline{\val(\V_{\k\ps_{conv}}(I))}$}.

The tropical variety $\TropC(I)$ is closely related to the initial ideals of the ideal $I$ as follows.  
For $\alpha \in \N^n$, let $x^{\alpha}$ denote $x_1^{\alpha_1}x_2^{\alpha_2}\cdots x_n^{\alpha_n}$. 
Given $w \in \R^n$ and a polynomial \mbox{$f(x) = \sum_{\alpha} f_{\alpha}x^{\alpha}$}, the initial form of $f$ with respect to 
$w$, denoted $\In_w(f)$, is the sum of the terms~\mbox{$f_{\alpha}x^{\alpha}$} which 
maximize~\mbox{$w\cdot \alpha$}. 
For an ideal $I\subset \C[x_1, \hdots, x_n]$, the initial ideal of $I$ with respect to~$w$, 
denoted~$\In_w (I)$, is
the ideal generated by the initial forms of elements of $I$, namely
$$\In_w(I)  = \langle  \In_w(f)  \ | \  f\in  I \rangle.$$

Suppose $y  \in \k\ps^n$ is an $n$-tuple of non-zero Puiseux series with 
valuation \mbox{$\val(y)\in \Q^n$} and leading 
coefficients $\coeff(y) \in (\k^*)^n$.  
If $y$ lies in the variety $\V_{\k\ps}(I)$ for some ideal 
\mbox{$I \subset \k[x_1, \hdots, x_n]$}, 
then $w = -\val(y)$  belongs to the tropical variety $\Trop_\k(I)$, 
and $\coeff(y)$ belongs to the variety of the initial ideal $\V_{\k}(\In_w(I))$. 
In particular, $\V_{\k}(\In_w(I))\cap(\k^*)^n$ is non-empty which implies that $\In_w(I)$ cannot contain a monomial. 
Over $\C$, the converse holds and is known as the fundamental theorem of tropical geometry.  
Indeed, for $\k=\C$, tropical varieties have very specific combinatorial structure, which we summarize from 
 \cite[Theorems~3.2.3~and~3.3.8]{MSbook}:

\begin{Thm}
Let $I\subset \C[x_1, \hdots, x_n]$ be a prime ideal that defines a 
$d$-dimensional variety~$\V_{\C^*}(I)$. 
Then, the tropical variety $\Trop_{\C}(I)$ is the support of a pure $d$-dimensional rational polyhedral fan.
It is the set of weights for which the initial ideal of $I$ contains no monomials: 
$\TropC(I) =   \{w\in \R^n~|~\In_w(I) \hbox{~does not contain a monomial}\}$.
 \end{Thm}

\begin{Remark} One can associate a multiplicity to each maximal cone in the tropical variety $\TropC(I)$. 
Each irreducible component of the variety $\V_{\C^*}(\In_w(I))$ contributes to the multiplicity of the ray $w$ 
by the multiplicity of the corresponding minimal prime in $\In_w(I)$. See \cite[\S 3.4]{MSbook} for details. 
With multiplicities, the tropical complex variety is \emph{balanced}.  In particular, if the tropical variety 
is a union of rays $\vec{r}_1, \hdots, \vec{r}_s$ with multiplicities $m_1, \hdots, m_s$, and $r_i$ is the primitive integer point on 
the ray $\vec{r}_i$, the weighted sum $m_1r_1 + \hdots + m_s r_s \in \R^n$ is zero.
\end{Remark}

As in classical algebraic geometry, real varieties provide subtle challenges for tropicalization.  
There are some similarities to the complex case. In particular, Alessandrini \cite{Alessandrini13} recently showed that $\TropR(I)$ 
is a rational polyhedral fan related to the real variety $\V_{\R^*}(I)$. 

\begin{Thm}[\cite{Alessandrini13}] For an ideal $I\in \R[x_1, \hdots, x_n]$, 
the real tropical variety $\TropR(I)$ is a rational polyhedral fan
in $\R^n$ whose dimension is at most the dimension of the variety $\V_{\R^*}(I)$.\end{Thm}

In general, we only know that $\TropR(I) \subseteq \TropC(I)$. 
In fact, $\TropR(I)$ is not necessarily a subfan of the Gr\"obner fan of $I$.  
For examples, see \cite[Fig. 7]{Alessandrini13} or \cite[Section 2]{Vinzant12}.
The real tropical variety $\Trop_{\R}(I)$ need not have the same dimension as $\V_{\R}(I)$, 
it need not satisfy any balancing conditions, and it may not be pure of any dimension.

\begin{Example} \label{ex:quintic}
Consider the sextic polynomial  $f  = x^6-x^3+y^2$.
The complex tropical variety $\TropC(f)$ consists of the rays spanned by $(-2,-3)$, $(1,3)$,  and $(0,-1)$, with multiplicities 1, 2, and 3, respectively, as 
shown in Figure~\ref{fig:quinticTrop}. 
In this case, we can use power series expansions to explicitly find solutions to $f=0$ over $\C\ps$ with these valuations, e.g.:
\begin{center}
\begin{tabular}{lcl}
$(x,y) = \left(t^2,\; t^3-\frac{t^9}{2}-\frac{t^{15}}{8}-\frac{t^{21}}{16} + \hdots  \right)$  & $ \ \ \rightarrow \ \ $ & $-\val(x,y) = (-2,-3) $ \\
$(x,y) =\left (\frac{1}{t}, \ -\frac{i}{t^3}+\frac{i}{2}+\frac{i t^3}{8}+\frac{i t^6}{16}+\hdots \right)$ &$ \ \ \rightarrow \ \ $  &$-\val(x,y) = (1,3) $ \\
$(x,y) = \left(1 -\frac{t^2}{3} -\frac{4 t^4}{9} -\frac{77 t^6}{81}+ \hdots , \  t \right)$ &$ \ \ \rightarrow \ \ $  &$-\val(x,y) = (0,-1).$ \\
\end{tabular}
\end{center}
Here, the Puiseux series $x,y\in \C\ps$ are locally convergent 
near $t = 0$ meaning that these tuples locally parametrize a path in $\V_{\C^*}(I)$. 
The first and last actually belong to $\R\ps$, 
and therefore contribute to rays in $\TropR(f)$. On the other hand, 
since the real variety $\V_{\R}(f)$ is compact, there is no point
in $\V_{\R\ps}(f)$ with valuation $(-1,-3)$. 
\end{Example}

\begin{center}
\begin{figure}[h] 
\begin{center}
\includegraphics[width=2in]{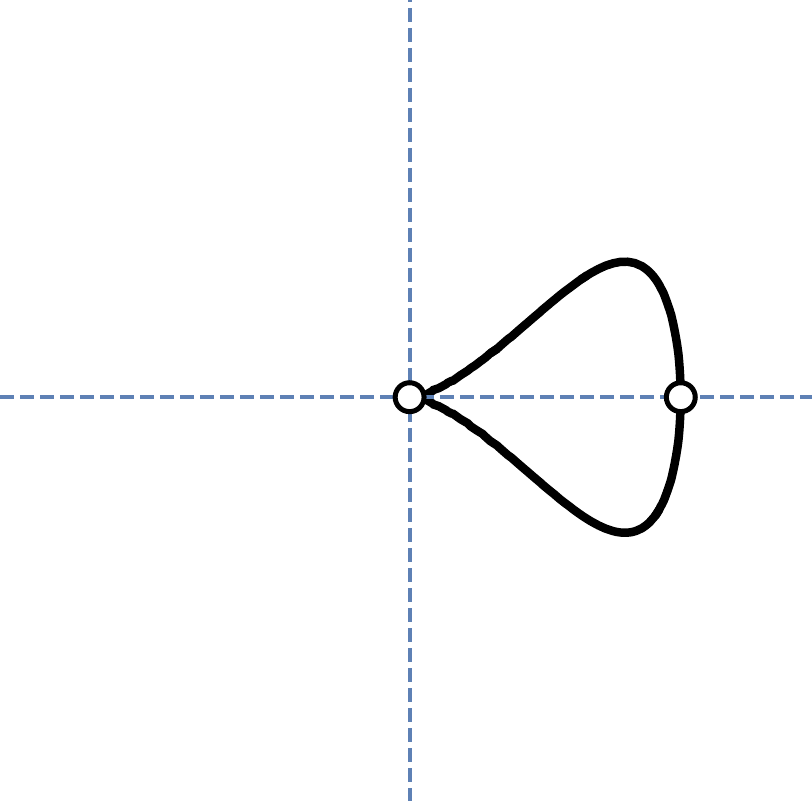} \quad \quad
\includegraphics[width=1.5in]{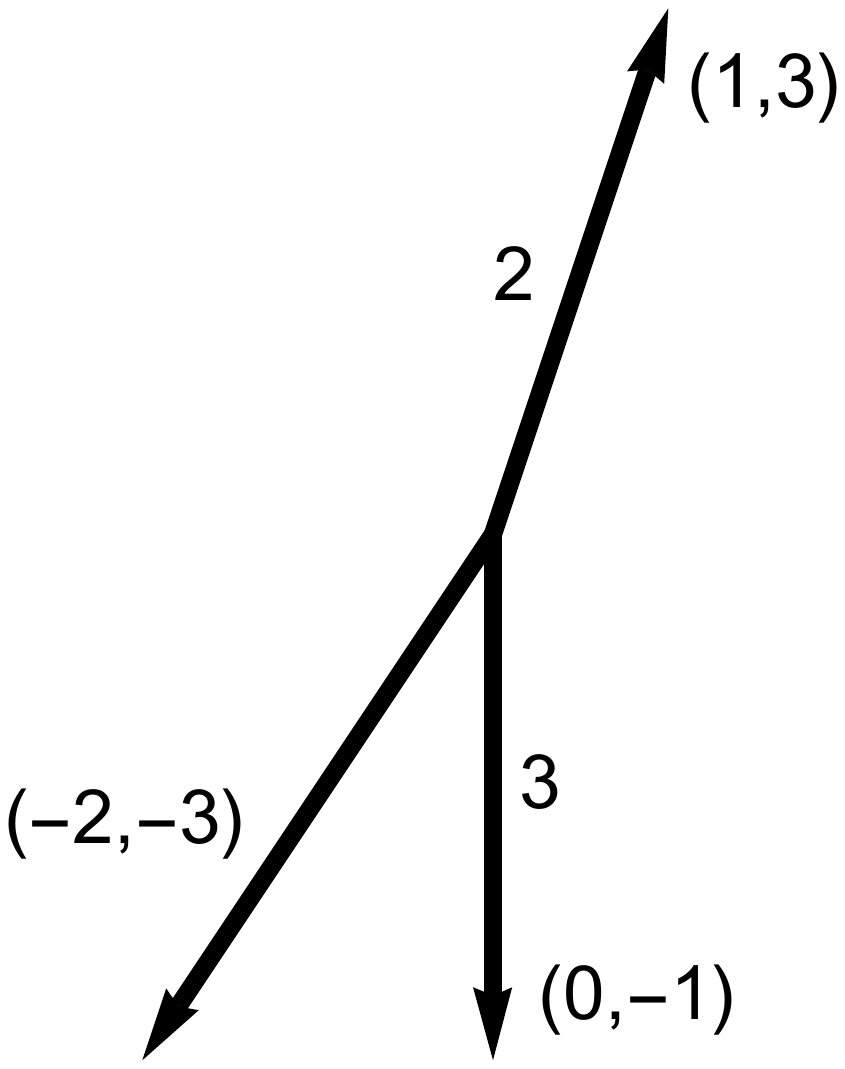} \quad \quad
\includegraphics[width=1.4in]{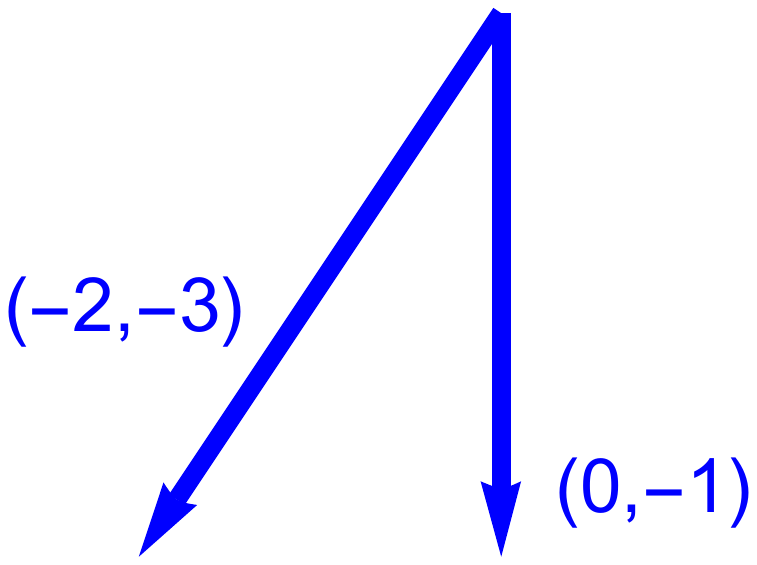}
\end{center}
\caption{The real points of a sextic plane curve,  its complex tropical variety with multiplicities listed when $> 1$, and its real tropical variety.}
\label{fig:quinticTrop}
\end{figure}
\end{center}

\begin{Remark}
In this paper, we avoid assigning multiplicities to rays in the real tropical variety, as it is not clear to us what the correct notion of multiplicity should be in this case.  
As the example below shows, it is not enough to count the multiplicity of 
real minimal primes associated to the initial ideal of a ray.  
However, in the process of computing a real tropical curve, 
one can compute the number of real Puiseux series with a given valuation $w$ 
that define different germs of a function around a point $p\in \V_{\R^*}(\In_w(I))$,
which may provide some notion of multiplicity in this case.
\end{Remark}

\begin{Example}
Consider the polynomial $f = ((x-1)^2+y^2) \cdot  ((x-1)^2-y^2) $. Its variety is the union of four lines, one real pair and one complex conjugate pair. 
The ray spanned by the vector $w=(0,-1)$ lies in both the real and complex tropical variety.  The corresponding initial form is $\In_w(f) = (x-1)^4$ 
so this ray has multiplicity four in $\TropC(f)$. Even though all four roots of $\In_w(f)$ are real, the contribution to the multiplicity from the real points is only two,
with the other contribution of two to the multiplicity arising from nonreal points.  
An irreducible polynomial with the same local behavior is $f+y^6$.
\end{Example}

\begin{Remark}
The real part of the complex torus $(\C^*)^n$ naturally breaks up into orthants with a given sign pattern.
Paths in $(\R^*)^n$ approaching $0$ or $\infty$ do so from within some orthant. 
In order to keep track of this data, consider the \emph{signed valuation map}
\[
{\rm sval}:\R\ps^* \rightarrow \Q \times \{+, -\} \ \ \ \text{ given by } \ \ \
{\rm sval}(x)  = 
\begin{cases}
(\val(x), +) & \text{if }\coeff(x)>0 \\
(\val(x), -) &  \text{if }\coeff(x)<0,    
\end{cases} \]
as in \cite{TropSimplex, Mikhalkin, Tabera, Viro08}.
The map ${\rm sval}$ extends coordinate-wise to $(\R\ps^*)^n $. For an ideal $I\subset\R[x_1, \hdots, x_n]$, 
the {\bf signed real tropical variety} ${\rm sTrop}_{\R}(I)$ is the image of the 
variety~$\V_{\R\ps^*}(I)$ under 
the map $-{\rm sval}$, where the negation acts on the vector in $\Q^n$ and fixes the sign pattern in $\{\pm\}^n$. 
The signed tropical variety lies in the disjoint union of $2^n$ copies of $\Q^n$.  
The {\bf positive tropical variety} ${\rm Trop}_{+}(I)$ is the part of the signed tropical variety 
lying in the positive copy, $\Q^n \times (+, \hdots, +)$, and has appeared 
in \cite{PositiveBergman, SpeyerWilliams05}, among others.
\end{Remark}

\begin{Example}\label{ex:butterfly}
Consider the quartic polynomial  $f  =  x^4 + y^4 - (x - y)^2 (x + y)$.
The real tropical variety $\TropR(f)$ consists of the rays spanned by the vectors $(0,-1)$, $(-1,0)$, and $(-1,-1)$. 
For example, the Puiseux series $(x,y) = (1 - t  -2 t^2 +\hdots , t)$ lies the variety of $f$. It image under $-{\rm sval}$
is $((0,-1),(+,+))$.  Replacing $t$ with $-t$ gives a point in $\R\ps^2$  in the variety of $f$ with $-{\rm sval}(x,y)=((0,-1),(+,-))$. 
The real variety of $f$ along with its signed tropical variety is shown in Figure~\ref{fig:butterfly}. 
We will return to this in Example~\ref{ex:butterfly2}. 
\end{Example}

\begin{center}
\begin{figure}[h] 
\begin{center}
\includegraphics[width=1.8in]{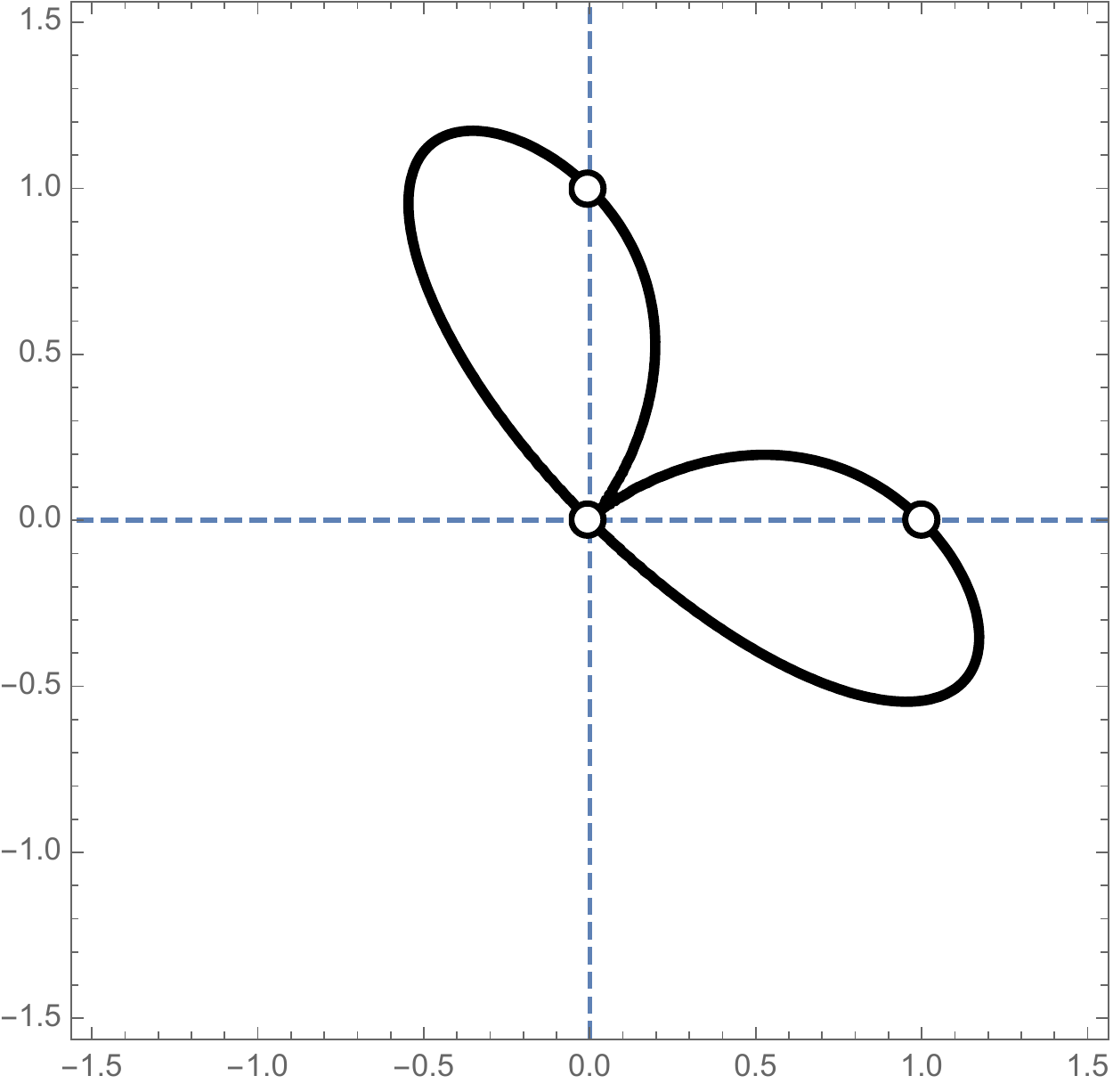} \quad 
\includegraphics[width=1.8in]{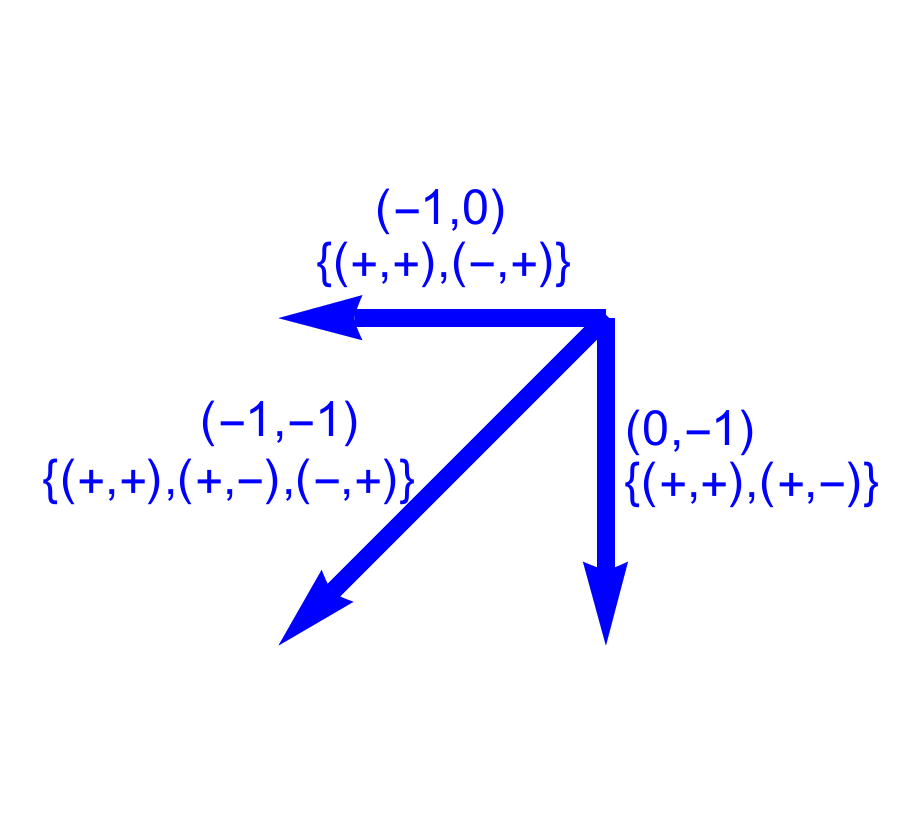}\quad
\includegraphics[width=1.8in]{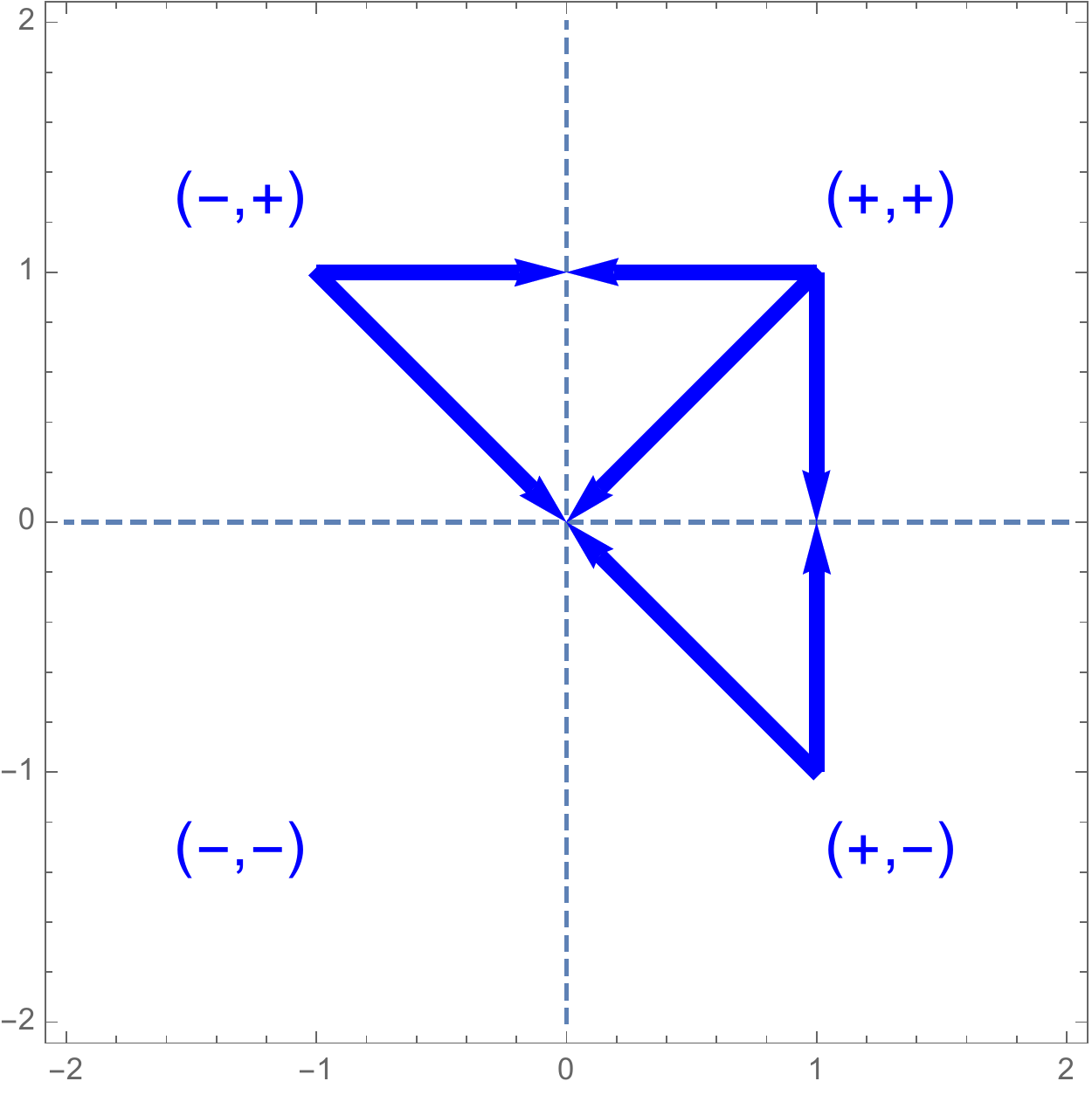}
\end{center}
\caption{The real points of a quartic plane curve with its signed tropical variety.}
\label{fig:butterfly}
\end{figure}
\end{center}

\section{Cauchy's integral formula and projectivization}\label{sec:Cauchy}

\subsection{Monodromy and Cauchy's integral formula}

Monodromy is related to Riemann surfaces, branch cuts, and singularities
in that one aims to study an object by walking around
in closed loops which can induce a non-trivial action.  
In the context of tropical curves, we will always reduce
down to the locally analytic case by performing a local uniformization
of complex curves, e.g., \cite[Thm.~A.3.2]{SW05}, after computing
the {\em cycle number}, which we now define for the case of interest. 
Let $C\subset \C^n$ be a curve not contained in a hyperplane 
and $p\in C$ whose $j^{\rm th}$ 
coordinate $p_j = \tau\neq 0$ is sufficiently small. 
This setup defines $P(s)\in C$ where $P(1) = p$ and $P_j(s) = \tau s$,
which is smooth for $0 < |s| \leq 1$.
For $\theta\geq0$, consider the continuous and periodic function
$\theta\mapsto P(e^{2\pi i\theta})$.
The {\bf cycle number} for the path $P(s)$ is the
minimum positive integer $c$ such that $P(e^{2\pi i c}) = p = P(1)$.
Since $C\cap\V(x_j - \tau)$ consists of at most $\deg C$ points,
it immediately follows that $1 \leq c \leq \deg C$.

In terms of Puiseux series, the cycle number $c$ is 
the minimum positive integer $n$ such that 
every coordinate of $P(s)$ lies in $\C((s^{1/n}))$.
That is, the cycle number $c$ is the minimum positive integer $n$
such that the function $s\mapsto P(s^n)$ is analytic on $0 < |s|\leq 1$.
In particular, if $\lim_{s\rightarrow0} P(s)\in\C^n$, 
then each coordinate of the function $s\mapsto P(s^c)$ is analytic on $|s|\leq 1$.

Suppose a function $f:\C\rightarrow\C$ is analytic on the unit disk in $\C$, specifically 
on and inside the closed loop $\gamma$ given by $\{e^{i\theta}~|~0\leq\theta\leq 2\pi\}$.
For $k\in\Z_{\geq0}$, Cauchy's integral formula yields
\begin{equation}
\frac{f^{(k)}(0)}{k!} \ \ = \ \ \frac{1}{2\pi i}\int_\gamma {\frac{f(z)}{z^{k+1}} dz} \ \ = \  \ \frac{1}{2\pi}\int_0^{2\pi} {\frac{f(e^{i\theta})}{(e^{i\theta})^{k}} d\theta}.  \label{eqn:cauchy_integral}
\end{equation}
In particular, $f$ has a power series expansion around $z=0$ and its valuation is the smallest $k$ for which $f^{(k)}(0)$ is non-zero. 
Using numerical path tracking, one can generate a discretization of $\gamma$ and numerically approximate $f^{(k)}(0)$. 

\subsection{Projectivization and affine patches}\label{sec:patching}

In addition to reducing down to the locally analytic case,
another key idea in our algorithm is to work on an affine patch 
for which the valuation of each point of interest is nonnegative.  
This corresponds with finite length solution paths that limit
to a coordinate~hyperplane. These paths can be parametrized
by tuples of Puiseux series with nonnegative valuations and therefore correspond to 
 points in the tropical variety with nonpositive coordinates. 
 
To that end, suppose that $C\subset(\C^*)^n$ is 
a curve defined by an ideal $I\subset\C[x_1,\dots,x_n]$.
In order to simplify the computations, we will first consider
the closure of $C$ in $\P^n$, namely
$$\overline{C} \ = \  \overline{\{[1:y_1: \hdots: y_n]~|~(y_1, \hdots, y_n)\in C\}} \ \subset \ \P^n.$$
The finitely many points in the boundary $\partial C = \overline{C}\backslash C$
are contained in the coordinate hyperplanes in~$\P^n$.
The ``hyperplane at infinity'' in $(\C^*)^n$ becomes one of these coordinate hyperplanes, namely $\{x_0=0\}$. 

We next restrict to
an affine coordinate patch.  That is, for a nonzero vector $v\in\C^{n+1}$,
consider the intersection $\widehat{C}$ of the affine cone over $\overline{C}$ with
the plane defined by $v\cdot x = 1$. Since the set of points of interest on $\overline{C}$ are its intersections with 
the coordinate hyperplanes, we require that each of these
finitely-many points correspond to a finite point in the coordinate patch defined
by~$v\cdot x = 1$,~i.e.,
\begin{equation} \label{eq:patch}
\{x\in \P^n~|~v\cdot x =0 \} \cap \{x \in \overline{C}~|\text{ some coordinate $x_i$ is zero }\}\ \ = \ \ \emptyset.
\end{equation}
If $d = \deg C$, there are at most $(n+1)d$ points in the intersection
of $\overline{C}$ with the union of the coordinate hyperplanes.
Hence, we must select $v\in\C^{n+1}$ outside of a hypersurface of degree 
at most $(n+1)d$.
In practice, we select the vector $v\in\C^{n+1}$ randomly.
When computing $\TropR$, the entries of the
vector $v$ should be real to maintain a relationship
between the real points of $C$ 
and the real points of $\widehat{C}$.
Under these hypotheses, 
the complex and real tropical varieties of $C$ can be recovered from those of $\widehat{C}$.

\begin{Prop}\label{prop:CtoCHat}If $v\in (\k^*)^{n+1}$ satisfies \eqref{eq:patch} where $\k = \R$ or $\C$, then 
$\Trop_{\k}(C)$ equals the image of $\Trop_{\k}(\widehat{C}) \cap (\R_{\leq 0})^{n+1}$ under the map 
$(w_0, \hdots, w_n) \mapsto (w_1 - w_0, \hdots, w_n - w_0)$.
Furthermore, for $\k=\C$, this map preserves multiplicities of rays from $\Trop_\C(\widehat{C})$ to $\Trop_\C(C)$. 
\end{Prop}

For the proof of this proposition below, we need two consequences of
 condition \eqref{eq:patch}.

\begin{Lemma}\label{lem:Zdiv} Let $\ell = v\cdot x\in \C[x_0, \hdots, x_n]_1$ where  $v\in (\C^*)^{n+1}$ satisfies \eqref{eq:patch}, and suppose that
$w\in \Trop(\overline{C}) \cap (\R_{\leq 0})^{n+1}$ with $w_j=0$ for some $j$ and $w_k\neq0$ for some $k$. 
Then the initial form $\In_w(\ell)$ does not vanish at any point $a\in (\C^*)^{n+1}$ in the variety of 
the initial ideal $\In_w(I(\overline{C}))$.
In particular, $\In_w(\ell)$ is not a zero-divisor modulo the ideal $\In_w(I(\overline{C}))\subset \C[x_0^{\pm1}, \hdots, x_n^{\pm1}]$. 
\end{Lemma}
\begin{proof} Let $J$ denote the homogeneous ideal $I(\overline{C})$, and suppose that $\In_w(\ell)$ is vanishes at some point $a = (a_0,\hdots, a_n) \in \V(\In_w(J)) \cap (\C^*)^{n+1}$.  
Since each coordinate of $w$ is non-positive and some coordinate is zero, the initial 
form of $\ell$ is the sum $\sum_{i\in \mathcal{A}}v_ix_i$ where $\mathcal{A} = \{i~|~ w_i=0\}$.  
Taking $e_i$ to be the $i$th coordinate vector, we see that $\ell( \sum_{i\in \mathcal{A}}a_ie_i)=\In_w(\ell)(a)  = 0$. 

On the other hand, by \cite[Prop. 3.2.11]{MSbook}, there exists a point $y\in \V_{\C\ps_{conv}}(J)$ 
with $-\val(y) = w$ and $\coeff(y) = a$. The leading terms of $y$ are then $(a_0t^{-w_0} , \hdots, a_nt^{-w_n})$, and
the Puiseux series $y_i(t)$ converges for small enough $t$.  Since each exponent $-w_i $ is nonnegative,
the coordinates $y_i(t)$ converge for $t=0$. Thus $y(0) =    \sum_{i\in \mathcal{A}}a_ie_i$ belongs to $\overline{C}$. 

Consider the point $p = \sum_{i\in \mathcal{A}}a_ie_i$. Since $w_j=0$ and $w_k\neq 0$, the point $p$ has $j^{\rm th}$ coordinate $a_j\neq0$ and 
$k^{\rm th}$ coordinate zero.  Also $\ell(p)=0$ and $p\in \overline{C}$. Thus $p$ belongs to the intersection in \eqref{eq:patch} and  $v$ does not satisfy
condition \eqref{eq:patch}. 
\end{proof}

\begin{Lemma}\label{lem:InSum}
Suppose that $J \subset \C[x_0, \hdots, x_n]$ is homogeneous ideal, 
$\ell \in \C[x_0, \hdots, x_n]_1$, and $w\in\R^{n+1}$.
If $\In_w(\ell)$ is not a zero-divisor modulo $\In_w(J)$, then 
\begin{equation}\label{eq:InitialIdealSum}
\In_w(J + \langle \ell-1\rangle) \ \ = \ \ \In_w(J )+\In_w( \langle \ell-1\rangle).
\end{equation}
\end{Lemma}
\begin{proof}
The containment $\supseteq$ is clear. To show the other containment, consider the set 
\[ S \ = \  \{(g,h) ~|~ \In_w(g + (\ell-1)h)\  \not\in \ \In_w(J )+\In_w( \langle \ell-1\rangle)\} \  \subset \  J\times \C[x_0, \hdots, x_n].\]
For the sake of contradiction, suppose that $S$ is non-empty and take $(g,h)\in S$ with minimal $\deg(h)$. 
Since $\In_w(g + (\ell-1)h)$ is not equal to  $\In_w(g) + \In_w((\ell-1)h)$, the sum of these leading terms must be zero.  
Taking further initial forms with respect to $\1 = (1,\hdots, 1)$ yields
$-\In_{\1}(\In_w(g)) \ = \ \In_{\1}(\In_w((\ell-1)h))  = \In_{\1}(\In_w(\ell-1))\cdot  \In_{\1}(\In_w(h)) $.
 
 If $\In_w(\ell - 1) = 1$, then the ideals on both sides of \eqref{eq:InitialIdealSum} are the ideal $\langle 1\rangle$ and thus equal.  
 Therefore we may assume that $\In_w(\ell)\neq1$, in which case $\In_\1(\In_w(\ell-1)) = \In_w(\ell)$. Putting this together with the arguments above, we see that 
 the product $\In_w(\ell)\cdot  \In_{\1}(\In_w(h))$ belongs to $\In_\1(\In_w(J))$. Because $J$ is homogeneous, the ideal $\In_w(J)$ is also homogeneous, 
by \cite[Lemma 2.4.2]{MSbook}. Hence $\In_\1(\In_w(J)) = \In_w(J)$. 
Thus, $\In_w(\ell)\cdot  \In_{\1}(\In_w(h))$ belongs to the ideal $\In_w(J)$. By assumption, $\In_w(\ell)$ is not a zero-divisor 
modulo $\In_w(J)$, so $\In_{\1}(\In_w(h))$ must belong to $\In_w(J)$. Therefore we can take $f\in J$ with $\In_w(f) = \In_{\1}(\In_w(h))$. 

Now consider $(g',h') = (g + (\ell-1)f, h-f)$. Since $f\in J$, $g + (\ell-1)f$ belongs to $J$.  Also, the sums $g'+(\ell-1)h'$ and $g+(\ell-1)h$ are equal, and therefore 
have equal initial forms. This shows that $(g',h')$ belongs to the set $S$.  On the other hand, because $\In_w(f) = \In_{\1}(\In_w(h))$,  $h' = h-f$ has strictly smaller degree than $h$, 
which contradicts our choice of $(g,h)$. 
\end{proof}

\begin{proof}[Proof of Proposition~\ref{prop:CtoCHat}]
We first prove the set-wise statement for $\k= \R, \C$. Let $K$ denote the field $\k\ps$. 
We use the notation $I = I(C)$, $\overline{I} = I(\overline{C})$, $\ell = v\cdot x$, and $\widehat{I} = \overline{I} + \langle \ell - 1\rangle  = I(\widehat{C})$. 
Because all the rays are defined over $\Q$, 
it suffices to prove that the two sets, $\Trop_\k(C)$ and the image of $\Trop_{\k}(\widehat{C} \cap (\R_{\leq 0})^{n+1})$, have the same points in $\Q^n$. 
   
($\supseteq$) Suppose $w \in \Trop_\k(\widehat{C}) \cap (\Q_{\leq 0})^{n+1}$.  Then there exists $y = (y_0, \hdots, y_n)\in \V_{K^*}( \widehat{I})$ 
with $\val(y) = -w$.  By homogeneity, the point $(1,y_1/y_0 ,\hdots, y_n/y_0)$ belongs to $\V_{K^*}(\overline{I})$, meaning that 
$(y_1/y_0 ,\hdots, y_n/y_0)$ belongs to $\V_{K^*}(I)$.  The negative of the valuation of this point is $(w_1-w_0, \hdots, w_n-w_0)$ and belongs to $\Trop_\k(C)$. 

($\subseteq$) Suppose $u\neq0$ belongs to $\Trop_\k(C) \cap \Q^n$.
Let $w = (0,u) - \max\{0, u_1, \hdots, u_n\} \1$. Then $w\in \Trop(\overline{C})\cap(\Q_{\leq 0})^{n+1}$ and $w$ has some zero coordinate from the attained maximum 
and some non-zero coordinate from non-equal coordinates of $w=(0,u)$. Lemma~\ref{lem:Zdiv} then implies that 
$\In_{w}(\ell)=\In_{(0,u)}(\ell)$ does not vanish any point in $\V(\In_w(\overline{I}))\cap(\C^*)^{n+1}$. 

Now take $y\in \V_{K^*}(I)$ with $-\val(y)=u$. Then $(1,y)$ belongs to $\V_{K^*}(\overline{I})$ and 
has $-\val(1,y) = (0,u)$ and $\coeff(1,y) = (1,a)$ for some $a\in (\C^*)^n$.  
In particular, $(1,a)$ belongs to the variety of the initial ideal $\In_{(0,u)}(\overline{I})$, meaning that 
$\In_{(0,u)}(\ell)(1,a)$ is not zero.  It follows that $\ell(1,y)$ equals $ct^q + \text{higher order terms}$, 
with leading coefficient $c=\In_{(0,u)}(\ell)(1,a)$ and valuation $q = \min\{0,\val(y_1), \hdots, \val(y_n)\} = -\max\{0,u_1, \hdots, u_n\}$. 

For every $\lambda\in K^*$, the point $(\lambda, \lambda y)$ belongs to $\V_K(\overline{I})$. Let $\lambda$ equal $1/\ell(1,y)$. 
Then the point $(\lambda, \lambda y)$ satisfies $\ell(\lambda, \lambda y) =1$ and belongs to $\V_K(\widehat{I})$. Furthermore
\[ \hbox{\small $-\val(\lambda, \lambda y) = -\val(1,y) - \val(\lambda) \1 = (0,u) +\val(\ell(1,y))\1 = (0,u) -\max\{0,u_1, \hdots, u_n\}\1 =w.$}  \]
Thus $w\in \Trop_\k(\widehat{C})$.  Also note that $w\in (\R_{\leq 0})^{n+1}$ and $u = (w_1-w_0, \hdots, w_n-w_0)$.

Now we fix $\k=\C$ and show that the multiplicity of $w$ in $\TropC(\widehat{C})$ equals the multiplicity of $(w_1-w_0,\hdots, w_n-w_0)$ in $\TropC(C)$. 
It is a general commutative algebra fact that if~$P$ is a minimal prime of a homogeneous ideal $J$ with multiplicity ${\rm mult}(P,J)=m$ 
and \mbox{$l(X)\in \C[x_0,\hdots, x_n]_1$}~is~not a zero-divisor modulo $J$, then $P+\langle l(x)-1\rangle$ is a minimal prime of $J + \langle l(x)-1\rangle$ 
with the same multiplicity $m$. 

Applying this with $J = \In_w(\overline{I})$ and $l(x)=x_0$ shows that the multiplicity of a ray $\vec{u}$ in $\TropC(I)$ equals the multiplicity 
of the cone $\R_+(0,u)+\R{\bf1}$ in $\TropC(\overline{I})$. On the other hand, by Lemmas~\ref{lem:Zdiv}~and~\ref{lem:InSum}, 
we can also apply this with $l(x) = \In_w(\ell)$. Note that since $w$ belongs to $(\R_{\leq 0})^{n+1}\cap \TropC(\ell-1)$, we have that $\In_w(\ell -1 ) = \In_w(\ell)-1$. 
This shows that the multiplicity of a ray $\vec{w}\in \Trop(\widehat{C})$ 
is that same as the multiplicity of $\vec{w} +\R\bf\1$ in $\TropC(\overline{I})$.  Putting these together, that a ray $w\in \TropC(\widehat{I})\cap(\R_{\leq 0})^{n+1}$ 
has the same multiplicity as the ray of $(w_1-w_0, \hdots, w_n-w_0)$ in $\TropC(I)$.
\end{proof}

This allows us to compute $\Trop_\k(C)$ by computing $\Trop_\k(\widehat{C})\cap(\R_{\leq 0})^{n+1}$. As described in the next section, 
the benefit of focusing on rays with nonpositive coordinates is that they correspond to paths in a variety with coordinates approaching zero, rather than infinity. 

\section{Algorithms}\label{sec:Algorithms}

The following describes algorithms for computing complex and real tropical curves.  
These algorithms use local uniformization of complex curves, e.g. \cite[Thm.~A.3.2]{SW05},
to transform Puiseux series into power series, which allows us to compute
valuations and multiplicities using Cauchy integrals. 
Given a curve $C\subset \C^n$ with corresponding curve $\widehat{C}\subset\C^{n+1}$ 
which is the intersection of the closure of $C$ in $\P^n$ and an affine
coordinate patch satisfying \eqref{eq:patch}, 
a rough outline of our strategy for computing its tropicalization is as follows: 
\begin{enumerate}
\item As defined in Section~\ref{sec:EOZ}, compute a point $\tau_j$ in the endgame operating zone of $\widehat{C}$ with respect to $x_j$ for $j=0,\hdots, n$.
\item As described in Section~\ref{sec:Val}, for each point $p \in \widehat{C}$ with some zero coordinate $p_j=0$ and continuous path 
$P(s)_{s\in[0,1]} \subset \widehat{C}$ such that $P_j(s) = \tau_j \cdot s$, compute an analytic reparametrization of $P(s)$ 
and use Cauchy integrals to compute the valuation of a power series expansion of each coordinate. 
\item  Collect all valuations with appropriate multiplicities to obtain $\Trop(C)$.
\end{enumerate}
The algorithms for computing $\TropC(C)$ (Section~\ref{sec:complex_algorithm}) and $\TropR(C)$ (Section~\ref{sec:real_tropical_algorithm}) 
rely on computations that are common to both.  For ease of exposition, we break these into separate algorithms, which are discussed in Section~\ref{subsec:InternalAlg}. 

\subsection{Algorithms common to $\C$ and $\R$}\label{subsec:InternalAlg}

There are two sub-algorithms which are common to both the complex and real 
algorithms.   First, we describe how to ensure that we are inside the radius
of convergence of the Puiseux series, which is 
called the {\em endgame operating zone}, e.g., see \cite[\S 10.3.1]{SW05}.
Once inside the endgame operating zone, we next describe computing
valuations and multiplicities via Cauchy's integral formula 
after computing a uniformization.  These two algorithms 
are essential in our approach for computing complex and real tropical curves.

\subsubsection{Finding the endgame operating zone}\label{sec:EOZ}

A main tool in the computations below is the parametrization of a curve $X\subset\C^{n+1}$ 
in a neighborhood of a point $p\in X$. 
One can parametrize a branch of the curve around this point by the value of 
some variable, say $x_j=s$.
Each of the other coordinates are then locally functions of $s$ that 
can be expressed as Puiseux series. 
To be meaningful, the computations below 
must take place in the radius of convergence of these Puiseux series.  
The domain of convergence of all the Puiseux series of the coordinates of a curve around a point is 
called the {\bf endgame operating zone}.  This zone is calculated  
by computing the critical points with respect to the 
parameterizing variable $x_j$.  Explicitly, 
suppose that $X$ is an irreducible curve not contained in any 
coordinate hyperplane and $f = (f_1,\hdots, f_{n}) \in (\k[x_0, \hdots, x_n])^{n}$ is a polynomial system 
such that $X$ is an irreducible component of the solution set of $f = 0$
which has multiplicity $1$ with respect to $f$. 
The set of critical points of $X$ with respect to $x_j$ and $f$ 
is the set of $x\in X$ for which $Jf(x)_{\widehat{j}}$
has a nonzero null vector.  Here we use $Jf(x)_{\widehat{j}}$ to denote the Jacobian matrix of $f$ evaluated
at $x$ with the $j^{\rm th}$ column removed.  

In Algorithm~\ref{alg:EOZ}, we actually compute a smaller threshold 
to simplify the computation of the real tropical curve, as we
will see in Section~\ref{sec:real_tropical_algorithm}.

\noindent\begin{algorithm}[h]
\SetKwInOut{Input}{Input}\SetKwInOut{Output}{Output}
\Input{An irreducible curve $X\subset\C^{n+1}$ not contained in 
a coordinate hyperplane and a polynomial system $f$
such that $X$ is an irreducible component of the solution set of $f = 0$
of multiplicity $1$; 
the set $\Lambda\subset\C^{n+1}$ consisting of points $X\cap\V(x_0\hdots x_n)$; 
index $j\in \{0,\hdots, n\}$.}
\Output{$\tau_j \in \R_{>0}$ such that the disk of radius $\tau_j$ centered at any point in $\Lambda$ 
is contained in the endgame operating zone of $X$ with respect to $x_j$.}
\BlankLine
\nl Compute the set $S$ of critical points of $X$ with respect to $x_j$ and $f$, i.e.,
$S = \{x\in X~|~Jf(x)_{\widehat{j}}\mbox{~has a nonzero null vector}\}$\;
\nl Compute $T_j = \{\abs(\pi_j(S))\} \cup \{\abs(\pi_j(\Lambda))\} \subset \R_{\geq 0}$
where $\pi_j(y_0, \hdots, y_n) = y_j$ \;
\nl Set $T_j^\ast = T_j\setminus\{0\}$ \;
\nl \Return $0 < \tau_j < \min(T_j^\ast)$, or some arbitrary positive number if $T_j^\ast$ is empty \; 
\BlankLine 
\caption{Computing $\tau_j$ inside the endgame operating zone of $X$ with respect to $x_j$.}
\label{alg:EOZ}
\end{algorithm}

The following is immediate from the definition of the endgame operating zone.

\begin{Prop}\label{Prop:EOZ}
Algorithm~\ref{alg:EOZ} returns $\tau_j > 0$ which is 
inside the endgame operating zone with respect to $x_j$.
\end{Prop}

\subsubsection{Computing valuations and multiplicities}\label{sec:Val}

Once inside the endgame operating zone, one can use path
tracking to compute valuations of a Puiseux series expansion.  
In Algorithm~\ref{alg:trop_rays}, we compute the primitive integer vector on the ray spanned by the valuation of this Puiseux series.
For a curve $X$, 
the input consists of a point $p\in X$ such that $p_j = \tau_j$
where $\tau_j$ is computed as in Algorithm~\ref{alg:EOZ}.
Then, we consider the path $P(s)\subset X$ parametrized
by $P_j(s) = \tau_j s$ for $s\in[0,1]$ with $P(1) = p$.  
In particular, this path $P(s)$ is a function of $s$ 
which corresponds with a convergent Puiseux series.  

\begin{algorithm}[h]
\SetKwInOut{Input}{Input}\SetKwInOut{Output}{Output}
\Input{An irreducible curve $X\subset\C^{n+1}$ not contained in 
a coordinate hyperplane; $j \in \{0,\hdots, n\}$;
$\tau \neq  0$ inside of the endgame operating zone of $X$ with respect to $j$; and
a point $p\in X$ such that $p_j = \tau$ which defines
the path $P(s)\in X\cap\V(x_j-\tau s)$ for $s\in[0,1]$ where $P(1) = p$.}
\Output{The primitive vector in $\Z_{\geq 0}^{n+1}$ of the ray spanned by the valuation of $P(s)$.}

\BlankLine
\nl Compute cycle number $c$ of the path $P(s)$ \;
\nl \For {$k$ from 0 to $n$} 
{
	 \nl Initialize with $u_k = -1$ and $a = 0$ \;  
	 \While {$a=0$} 
		{
		\nl Update $u_k \gets u_k+1$ \;
		\nl Update $a\gets \int_{0}^{2\pi}P_k(e^{c\cdot i\theta})/(e^{i\theta})^{u_k}d\theta$ \;
		}
}
\nl Set $u = (u_0, \hdots, u_n) \in  \Z_{\geq 0}^{n+1}$ \; 
\nl Compute $g =\gcd(u)$ \;
\nl \Return Primitive vector $r =  u/g \in \Z_{\geq 0}^{n+1} $ \;
\BlankLine 
\caption{Computing the primitive vector corresponding to a path valuation.}
\label{alg:trop_rays}
\end{algorithm}

The following shows that the path $P(s)$ is parametrized by a Puiseux series and that Algorithm~\ref{alg:trop_rays} correctly computes its valuation. 
 
\begin{Prop}\label{Prop:Path2Puiseux}
Let $p\in X \subset \C^{n+1}$, $j$,  and $\tau$ be the input of Algorithm~\ref{alg:trop_rays} and $r\in \Z_{\geq 0}^{n+1}$ be the output.
Then there is a point $y\in (\C\ps)^{n+1}$ in $\V_{\C\ps}(I(X))$ that converges for $t\in [0,1]$
and satisfies $y(s) = P(s)$ for all $s\in [0,1]$.  The valuation of $y$ is $(1/r_j) \cdot r\in \Q^{n+1}_{\geq 0}$. 
\end{Prop}

\begin{proof}
Since $c>0$ is the cycle number, local uniformization, e.g., \cite[Thm.~A.3.2]{SW05},
yields that $s\mapsto P(s^c)$ is analytic with respect to~$s$ on a neighborhood of $s=0$. 
Because $\tau$ was chosen inside the endgame operating zone of $C$ with respect to $x_j$, 
this neighborhood contains the unit disk $\{z\in \C~|~|z|\leq 1\}$. 
Taking power series expansions of each coordinate gives $\widetilde{y}\in (\k[[t]]_{conv})^{n+1}$ 
with $\widetilde{y}(s) = P(s^c)$ for all $|s|\leq 1$. 

We claim that $u$ is the valuation of $\widetilde{y}$.  To see this, write the $k$th coordinate as 
$\widetilde{y}_k(t) = \sum_\ell a_\ell  \ t^\ell $ where $a_\ell \in \C$. Then for any $\ell\in \Z_{\geq 0}$, 
the integral $ \int_{0}^{2\pi}P_k(e^{c\cdot i\theta})/e^{i \ell \theta}d\theta $
equals the integral $ \int_{0}^{2\pi}\widetilde{y}_k(e^{i\theta})/e^{i \ell \theta}d\theta $.  By the Cauchy integral formula
this integral equals $a_{\ell}/2\pi$.  Algorithm~\ref{alg:trop_rays} computes $u_k = \val(\widetilde{y}_k)$ as the minimum 
$\ell$ for which this integral is nonzero. 

Now set $y(t) = \widetilde{y}(t^{1/c})\in \C\ps^{n+1}$.  Then $\val(y) = (1/c) \val(\widetilde{y}) = (1/c)u$. 
Since $\widetilde{y}$ converges for $t\in [0,1]$, so does $y$. 
Furthermore, for $s\in [0,1]$, we have $y(s) = \widetilde{y}(s^{1/c}) = P(s)$. 
In particular, 
the $j^{\rm th}$ coordinate of $y$ is $y_j = \tau t$ and $\val(y_j) = 1$.  
The valuation of $y$ is a multiple of $u$ with $j^{\rm th}$ coordinate 1, meaning that $\val(y) = (1/r_j)\cdot r$. 
\end{proof}

Furthermore, any Puiseux series $y\in \V_{\k\ps_{conv}}(I(X))$ with non-negative valuations
and~$j^{\rm th}$ coordinate $y_j = \tau t$ parametrizes such a path in $X$. 

\begin{Prop}\label{Prop:Puiseux2Path}
Suppose that $y\in \V_{\k\ps_{conv}}(I(X) + \langle x_j - \tau t\rangle)$ and $\val(y)\in (\Q_{\geq 0})^{n+1}$. 
Then~$y$ converges for $t\in [0,1]$ and $p = y(1)$ defines a viable input to Algorithm~\ref{alg:trop_rays}. 
\end{Prop}

\begin{proof}
Since $y$ is convergent in a neighborhood of $t=0$ and $\tau$ was chosen inside the endgame operating zone 
of $X$ with respect to $x_j$, we know $y$ converges for $s\in (0,1]$. 
We see immediately that $p = y(1)$ has $j^{\rm th}$ coordinate $p_j = y_j(1) = \tau$. 
Since each coordinate $y_k$ has nonnegative valuation, it has a well-defined value at $t=0$. 
Therefore $P(s) = y(s)$ defines a continuous path 
in $X\cap \V(\langle x_j - \tau s\rangle)$ for $s\in [0,1]$.
\end{proof}

%
%
\subsection{Computing complex tropical curves}
\label{sec:complex_algorithm}

By using Algorithms~\ref{alg:EOZ} and~\ref{alg:trop_rays},
we now present an approach for computing complex tropical
curves of an irreducible curve $C\subset(\C^*)^n$, namely Algorithm~\ref{alg:TropC}.
We assume that we are given a polynomial system $f$
such that $C$ is an irreducible component of the solution set
of $f = 0$ which has multiplicity~$1$ with respect to $f$.
As discussed in Section~\ref{sec:patching}, we also 
start with the irreducible curve $\widehat{C}\subset\C^{n+1}$ 
with corresponding polynomial system $\widehat{f}$, i.e.,
$\widehat{C}$ is an irreducible component of the solution set
of $\widehat{f} = 0$ which has multiplicity $1$ with respect to~$\widehat{f}$.

\noindent \begin{algorithm}[ht]
\SetKwInOut{Input}{Input}\SetKwInOut{Output}{Output}
\Input{An irreducible curve $C\subset(\C^*)^{n}$; a polynomial 
system $f$ such that $C$ is an irreducible
component of multiplicity one of the solution set $f = 0$;
and a vector 
$v\in (\C^*)^{n+1}$ satisfying condition \eqref{eq:patch}.}
\Output{$\TropC(C)$, a collection of primitive vectors in $\Z^n$ with multiplicity.}
\BlankLine
\nl Initialize $\TropC(C) = \emptyset$ \; 
\nl Define $\widehat{C} = \overline{C}\cap\V(v\cdot x -1)$ and $\widehat{f} = \overline{f}\cup\{v\cdot x-1\}$ \;
\nl Set $\Lambda =\widehat{C}\cap\V(x_0\cdot x_1\cdots x_n)$\;
\nl Partition $\Lambda = \sqcup_j \, \Lambda_j$, where $\Lambda_j = \{ x \in \Lambda~|~x_0,\ldots,x_{j-1} \neq 0, x_j = 0 \}$ \;
\nl \For {$j$ from 0 to $n$} 
{
	\nl \If {$\Lambda_j = \emptyset $}
	{
	\nl Continue \; 
	}
		\nl Call Algorithm~\ref{alg:EOZ} using $\widehat{C}$, $\widehat{f}$, $\Lambda$, and $j$
	yielding $\tau_j > 0$ \;	
	\nl Compute $C_j^\tau = \widehat{C} \cap \V(x_j - \tau_j)\subset\C^{n+1}$ \;
	\nl \nosemic Compute $\Lambda_j^\tau$ which consists of $p\in C_j^\tau$ such that the path starting, with $s=1$, at $p$ \;
	 \pushline \dosemic defined by $\widehat{C}\cap\V(x_j - \tau_j s)$ for $s\in[0,1]$ 
	 ends at a point in $\Lambda_j$ \;
	 \popline 
	 \nl \For {each $p \in \Lambda_j^\tau$}
	{
		\nl Call Algorithm~\ref{alg:trop_rays} using $\widehat{C}$, $j$, $\tau_j$, and $p$ yielding $r\in \Z_{\geq 0}^{n+1}$\;
		\nl Add $(r_0-r_1,\dots,r_0-r_n)$ to $\TropC(C)$ with multiplicity contribution $1/r_j$\;
	}
}
\nl \Return $\TropC(C)$ \;
\BlankLine 
\caption{Computation of $\TropC$}\label{alg:TropC}
\end{algorithm}
\smallskip 

Before proving that Algorithm~\ref{alg:TropC} computes the tropical
curve of $C$, we consider an illustrative example that demonstrates
the steps of the algorithm.

\begin{Example}\label{ex:IllusEx}
Consider the quartic curve $C \subset \C^2$ defined by $f = x_1^3 x_2 - x_1x_2^3 + x_1^3  - x_2^2  =0$, shown in Figure~\ref{fig:patches}. 
Its homogenization $\overline{f} =x_1^3 x_2 - x_1x_2^3 + x_0x_1^3  - x_0^2x_2^2   $ defines a curve $\overline{C}\subset \P^2$, 
and the vector $v=(1,1,2)$  satisfies the condition   \eqref{eq:patch}. 
Thus, $\widehat{C} \subset \C^3$ is the curve 
defined by $\overline{f}=0$ and $x_0+x_1+2x_2=1$.
The set $\Lambda$ consists of the five intersection points of $\widehat{C}$ with $\V(x_0x_1x_2)$. We partition the points by their first zero coordinate:
\[\Lambda_0 = \{(0, 1, 0), (0, 1/3, 1/3), (0, 0, 1/2), (0, -1, 1)\}, \ \ \ \Lambda_1 = \{(1,0,0)\}, \ \text{ and } \ \Lambda_2 = \emptyset.\]
For each $j=0,1,2$, Algorithm~\ref{alg:EOZ} computes $\tau_j>0$ inside the endgame operating zone
by calculating the projections of the points $\Lambda$ onto each coordinate and 
the values $t\in\C$ at which the plane $x_j-t$ is tangent to the curve $\widehat{C}$. 
For $j=0$ the minimum of non-zero absolute value of these numbers is $\min(T_0^*) \approx 0.2162$. 
Indeed, as seen in Figure~\ref{fig:patches}, the plane $x_0=-\min(T_0^*)$ is tangent to $\widehat{C}$. 
For $j=1$, we find $\min(T_1^*) \approx   0.2483$, which is attained by a complex tangency.
We take $\tau_j$ to be any positive number less that $\min(T_j^*)$, for example $\tau_0=\tau_1= 0.1$.  Then, the paths defined by the intersection of $\widehat{C}$ and a
disk of radius~$\tau_j$ around each point in $\Lambda_j$ correspond with 
convergent Puiseux series.  

For each $j=0,1,2$, we track points limiting to the points in $\Lambda_j\subset \V(x_j)$. 
For example, for $j=0$, we compute $C_0^{\tau} = \widehat{C}\cap\V(x_0 - 0.1)$, which consists of four points in $\C^{3}$, each of which tracks to a unique point in 
$\Lambda_0$. Therefore $C_0^{\tau}  = \Lambda_0^{\tau}$.
We apply Algorithm~\ref{alg:trop_rays} to each point $p\in \Lambda_0^{\tau}$. 
The point $p\approx(0.1, -0.022, 0.461)\in \Lambda_0^{\tau}$ defines a path 
$P(s)\in\widehat{C}\cap \V(x_0-0.1 s)$ with $P(1) = p$ and $P(0) = (0,0,1/2)$. The cycle number of this path is $c=1$, meaning that the map 
$z \mapsto P(z)$ is analytic for $|z|\leq 1$.  In particular, each coordinate $g_j(z) = P_j(z)$ has a power series expansion 
in $z$. Using Cauchy's integral formula, we can find the leading term in this power series. By definition $g_0(z) =0.1 z$. For $g_1(z) = P_1(z)$, 
we use numerical integration to calculate that 
$$\hbox{\small $\displaystyle \frac{g_1^{(k)}(0)}{k!} = \frac{1}{2\pi} \int_0^{2\pi} \frac{g_1(0.1 e^{i\theta})}{(0.1 e^{i\theta})^{k}} d\theta = 0 \ \text{ for $k=0,1$, and  }  \ \frac{g_1^{(2)}(0)}{2!} = \frac{1}{2\pi} \int_0^{2\pi} \frac{g_1(0.1 e^{i\theta})}{(0.1 e^{i\theta})^{2}} d\theta = -0.02$}.$$
Therefore, the leading term of the power series expansion of $g_1(z)$ is $\frac{g_1^{(2)}(0)}{2!}z^2 = -0.02z^2$. Similarly, we see that 
$g_2(z) = 1/2 + \text{higher degree terms}$.
Therefore Algorithm~\ref{alg:trop_rays} outputs $r=(1,2,0)$.  This contributes $(r_0-r_1, r_0-r_2) = (-1,1)$
to $\TropC(C)$ with multiplicity $1$. 

On the other hand, for $j=1$, $C_1^{\tau} = \widehat{C}\cap\V(x_1 - 0.1)$ consists of four points in $\C^{3}$, 
two of which are complex and limit to the point $(0,0,1/2) \in \Lambda_0$, two of which are real and 
track to $(1,0,0)\in \Lambda_1$. The latter two are $\Lambda_1^{\tau}$. The point $p\approx(0.8293, 0.1, 0.0354)\in \Lambda_1^{\tau}$ defines a path 
$P(s)\in\widehat{C}\cap \V(x_1-0.1 s)$ with $P(1) = p$ and $P(0) = (1,0,0)$. The cycle number of this path is $c= 2$, meaning that the map 
$z \mapsto P(z^2)$ is analytic for $|z^2|\leq 1$.  As above we use Cauchy's integral formula to 
compute power series expansions of each coordinate $g_j(z) = P_j(z^2)$, to find 
$P(z^2) = (1, 0.1z^2, (0.1)^{3/2} z^3) + \text{higher degree terms}$, giving $r = (0,2,3)$. 
This contributes $(r_0-r_1, r_0-r_2) = (-2,-3)$ to $\TropC(C)$ with multiplicity $1/r_1 = 1/2$.
Repeating this process for the second point $p\approx(0.9635, 0.1, -0.0317)\in \Lambda_1^{\tau}$ also gives the ray $(-2,-3)$ 
with multiplicity $1/2$, for total multiplicity of $1$.  The other rays in $\TropC(C)$ are computed similarly, giving the tropical curve in Figure~\ref{fig:patches}. 
\end{Example}

\begin{center}
\begin{figure}
\begin{center}
\includegraphics[width=1.8in]{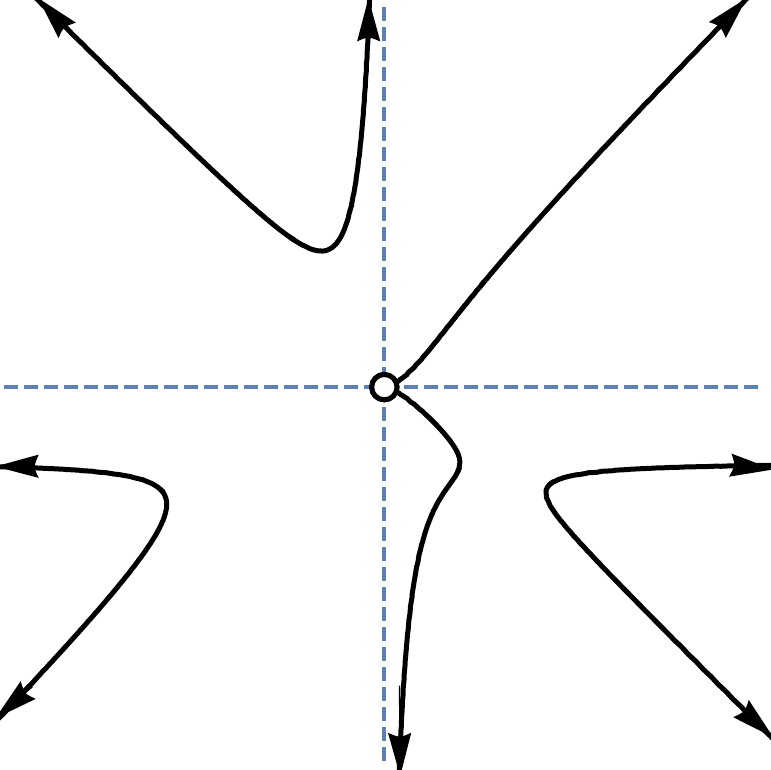} \quad \quad 
\includegraphics[width=1.8in]{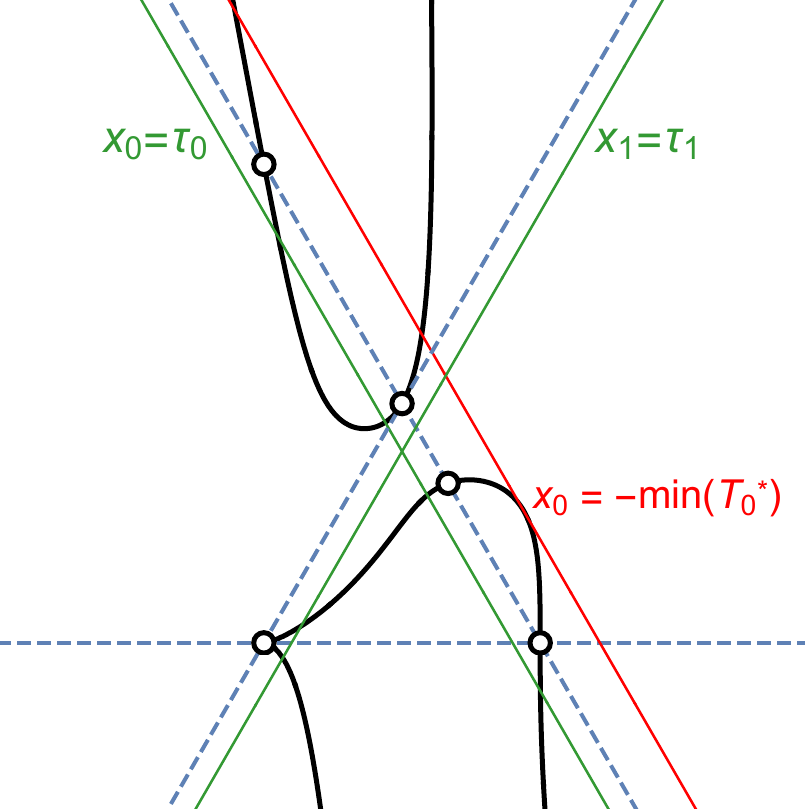}\quad \quad 
\includegraphics[width=1.6in]{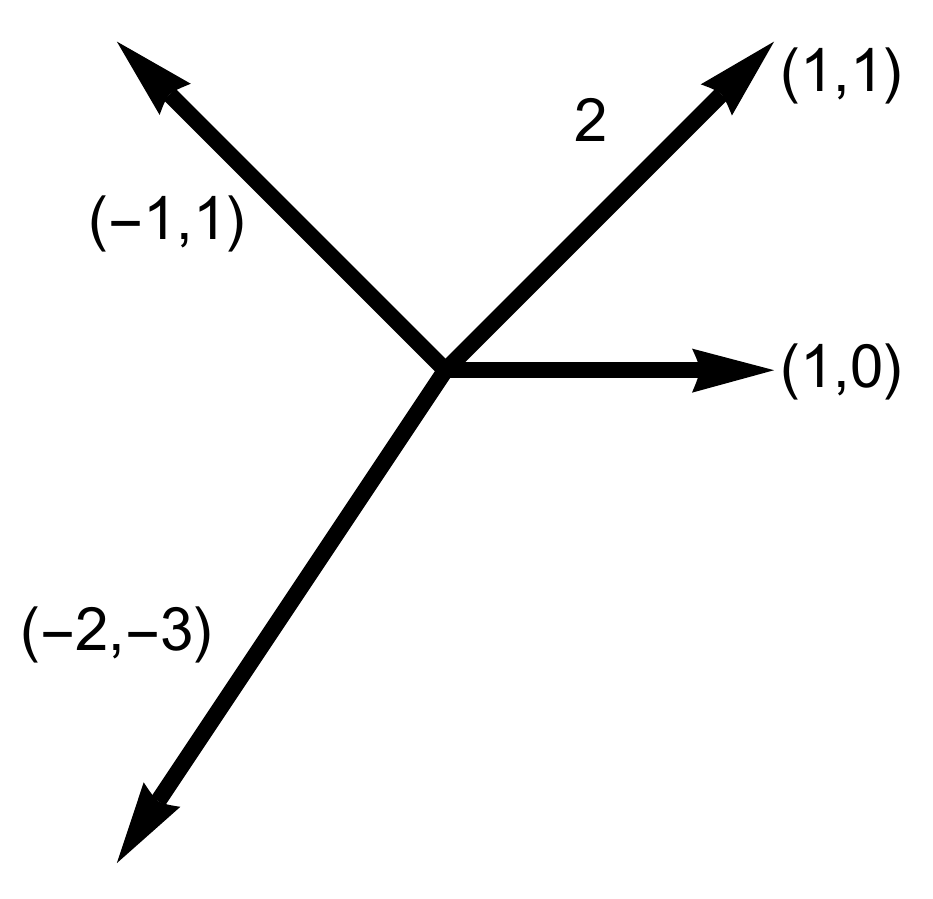}
\end{center}
\caption{Affine patches $C$ and $\widehat{C}$ of the quartic curve in Example~\ref{ex:IllusEx} and $\TropC(C)$.}\label{fig:patches}
\end{figure}
\end{center}

\begin{Thm}\label{thm:TropC}
Algorithm~\ref{alg:TropC} computes the primitive vectors with multiplicities of $\TropC(C)$.
\end{Thm}
\begin{proof}
Let $\mathcal{T}$ be the set of vectors with multiplicities computed by Algorithm~\ref{alg:TropC}.
Suppose that $u$ is a primitive vector of $\TropC(C)$ with multiplicity $m$. 
By Proposition~\ref{prop:CtoCHat}, we know 
$w = (0,u) - \max\{0, u_1,\hdots,u_n\}\1$ is a primitive
vector of $\TropC(\widehat{C})$ with multiplicity $m$. 
Let~$j$ be the first index for which $w_j<0$ and let $\tau_j$ be the output of Algorithm~\ref{alg:EOZ} with $X=\widehat{C}$. 
The ray $\vec{w}$ and tropical hyperplane $\TropC(\langle x_j - \tau_j t\rangle)$ meet transversely 
at the point $ |1/w_j |w$. By \cite[Def.~3.6.11]{MSbook}, the multiplicity of $|1/w_j |w$ in $\TropC(\widehat{I} + \langle x_j - \tau_j t\rangle)$
equals $m\cdot |w_j|$.

Let $K = \C\ps_{conv}$. By \cite[Prop.~3.4.8]{MSbook}, the multiplicity of $|1/w_j |w$ in 
the tropical variety of $\widehat{I} + \langle x_j - \tau_j t\rangle$ equals the number of  
points in its variety over $K$ with this valuation,
\[V_w  \ \ = \ \  \left.\left\{y\in \V_K(\widehat{I} + \langle x_j - \tau_j t\rangle ) ~\right|~| \val(y)  =   -|1/w_j |w\right\}.\]
By Proposition~\ref{Prop:Puiseux2Path}, any $y\in V_{w}$ converges for $t\in [0,1]$ 
and $p=y(1)$ defines a viable 
input for Algorithm~\ref{alg:trop_rays}.
Since $y_j = \tau_j t$, we see that $p\in C^{\tau}_j = \widehat{C}\cap \V(x_j - \tau_j)$.  Furthermore, since 
$j$ is the smallest index for which $\val(y_j)>0$, $j$ is the smallest zero coordinate 
of $y(0)$. Therefore $y(0)$ belongs to $\Lambda_j$ and $p$ belongs to $\Lambda_j^{\tau}$.
By Proposition~\ref{Prop:Path2Puiseux}, Algorithm~\ref{alg:trop_rays} returns primitive vector $r$, where 
$(1/r_j)r = -|1/w_j |w$ is the valuation of $y$.  Thus, each point in $V_w$ contributes 
the vector $u = (r_0-r_1, \hdots, r_0-r_n)$ to $\mathcal{T}$ with multiplicity $1/r_j = 1/|w_j|$.  
As $\#V_w = m\cdot |w_j|$, the 
vector $u$ appears in $\mathcal{T}$ with multiplicity at least $m$. 
By Proposition~\ref{Prop:Path2Puiseux}, every path used in Algorithm~\ref{alg:trop_rays} 
comes from some tuple of Puiseux series, giving us equality. 
\end{proof}

\begin{Example}\label{ex:butterfly2} The curve from Example~\ref{ex:butterfly} displayed in Figure~\ref{fig:butterfly} demonstrates some of the subtlety in 
computing the multiplicities of rays in the tropical variety. 
The cusp at the origin contributes to the ray $(-1,-1)$ in $\TropC(f)$ with multiplicity $2$.

There are two points in $\V(f)\cap\{x=\tau_j\}$ that track to the origin along this cusp, and each of these paths has cycle number $c=2$. 
Re-parametrizing by $x=t^2$ and using Cauchy's integral formula, we find the initial terms of the two corresponding Puiseux series in $\V_{\R\ps}(f)$:
\[ ( x,y) \ \ = \ \  \left( t^2, \ \ t^2  \ \pm \ t^3\ +\ \frac{3 }{4}\ t^4\ \pm\ \frac{45 }{32}\ t^5   \ + 
\ \hdots \right), \]
each with $u = \val(1,x,y) = (0,2,2)$. Thus, $g =\gcd(u) = 2$. 
For each of these two paths, Algorithm~\ref{alg:trop_rays} returns primitive vector $r= u/g = (0,1,1)$. 
In Algorithm~\ref{alg:TropC}, each of these paths contributes $(-1,-1)$, each with multiplicity 1, thereby yielding a total contribution of~$2$.
\end{Example}

\begin{Remark}
We highlight some of the key differences between Algorithm~\ref{alg:TropC}
and the approach presented \cite{jensen2014computing}.  
First, Algorithm~\ref{alg:TropC} explicitly computes the endgame operating zone 
to ensure all computations are performed on 
convergent Puiseux series.  The approach of 
\cite{jensen2014computing} requires slicing so that the intersection
points are within a so-called tentacle.  Then, they track these
to compute additional points in the same tentacle.
By using lattice recovery techniques, they heuristically recover the valuation.  
Although our implementation as described in 
Section~\ref{sec:implementation}
performs computations which are not certified, 
these computations are amendable to certification.  
In our approach, inside the endgame operating
zone, we first compute the cycle number.
Since this can be performed using a Newton homotopy,
this computation can be certifiably computed using \cite{CertifiedTrack1,CertifiedTrack2}.
Then, Cauchy's integral theorem is used to compute the valuation.  
When the Cauchy integral cannot be computed exactly, the use of the trapezoid rule
produces an exponentially convergent numerical 
method \cite{trefethen2014exponentially} for computing the Cauchy integral.  

The techniques of \cite{jensen2014computing} could possibly be adapted to compute real tropical curves, but 
that is not pursued by the authors.   Algorithm~\ref{alg:TropC} can be
transformed into a method for computing the real tropical
curve, as described below.
\end{Remark}

%
%

\subsection{Computing real tropical curves}
\label{sec:real_tropical_algorithm}

Computing the real tropical curves is similar
to that of complex tropical curves except that only paths 
starting at real points are considered, as presented in Algorithm~\ref{alg:TropR}. 

\noindent \begin{algorithm}[!hb]
\SetKwInOut{Input}{Input}\SetKwInOut{Output}{Output}

\Input{A real irreducible curve $C\subset(\C^*)^{n}$ 
with $C\cap(\R^*)^{n}\neq\emptyset$; a polynomial 
system $f$ such that $C$ is an irreducible
component of multiplicity one of the solution set $f = 0$;
and a vector $v\in (\R^*)^{n+1}$ satisfying condition \eqref{eq:patch}.}
\Output{$\TropR(C)$, a collection of primitive vectors in $\Z^n$ with signs $\{\pm 1\}^n$.}
\BlankLine	
\nl Define $\widehat{C} = \overline{C}\cap\V(v\cdot x -1)$ and $\widehat{f} = \overline{f}\cup\{v\cdot x-1\}$ \;
\nl Set $\Lambda =\widehat{C}\cap\V(x_0\cdot x_1\cdots x_n)$\;
\nl \If {$\Lambda = \emptyset$}
{
  \nl \Return $\TropR(C) = \{0\}$ \;
}
\nl Initialize $\TropR(C) = \emptyset$ \; 
\nl Partition $\Lambda = \sqcup_j \, \Lambda_j$, where $\Lambda_j = \{ x \in \Lambda~|~x_0,\ldots,x_{j-1} \neq 0, x_j = 0 \}$ \;
\nl \For {$j$ from 0 to $n$} 
{
	\nl \If {$\Lambda_j = \emptyset $}
	{
	\nl Continue \; 
	}
	\nl Call Algorithm~\ref{alg:EOZ} using $\widehat{C}$, $\widehat{f}$, $\Lambda$, and $j$ yielding $\tau_j > 0$ \;
	\nl Compute $C_j^+ = \widehat{C} \cap \V(x_j - \tau_j) \cap \R^{n+1}$ and $C_j^- = \widehat{C} \cap \V(x_j + \tau_j) \cap \R^{n+1}$ \;
	\nl \nosemic Compute $\Lambda_j^{+}$ which consists of $p\in C_j^+$ such that the path starting with $s = 1$\;
	\pushline \dosemic at~$p$ defined by $\widehat{C}\cap\V(x_j-\tau_j s)$ for $s\in[0,1]$ ends at a point in $\Lambda_j$ \;
	\popline \nl \For {each $p \in \Lambda_j^+$}
	{
		\nl \dosemic Call Algorithm~\ref{alg:trop_rays} using $\widehat{C}$, $j$, $\tau_j$, and $p$ yielding $r\in (\Z_{\geq 0})^{n+1}$\;
		\nl \nosemic Add $(r_0-r_1,\dots,r_0-r_n)$ to $\TropR(C)$ with sign $ (\sign(p_0p_1),\dots,\sign(p_0p_n))$ \;
	}	
	\nl \nosemic Compute $\Lambda_j^{-}$ which consists of $p\in C_j^-$ such that the path starting with $s = 1$\;
	\pushline \dosemic at~$p$ defined by $\widehat{C}\cap\V(x_j+\tau_js)$ for $s\in[0,1]$ ends at a point in $\Lambda_j$ \;
	\popline \nl \For {each $p \in \Lambda_j^-$}
	{
		\nl Call Algorithm~\ref{alg:trop_rays} using $\widehat{C}$, $j$, $-\tau_j$, and $p$ yielding  $r\in (\Z_{\geq 0})^{n+1}$ \;
	\nl \nosemic Add $(r_0-r_1,\dots,r_0-r_n)$ to $\TropR(C)$ with sign $(\sign(p_0p_1),\dots,\sign(p_0p_n))$ \;
	}
}
\nl \Return $\TropR(C)$ \; 
\BlankLine 
\caption{Computation of $\TropR$.}\label{alg:TropR}
\end{algorithm}

Before proving that Algorithm~\ref{alg:TropR} computes the real
tropical curve of $C$, we consider an illustrative example that demonstrates
the steps of the algorithm.

\begin{Example}\label{ex:IllusExReal}
Consider the quartic curve $C$ of Example~\ref{ex:IllusEx}. The computation of $\TropR(C)$ 
largely follows that of $\TropC(C)$.  We now must check both sides of the coordinate hyperplanes $\V(x_j)$ for real points.  
For example, intersecting  $\widehat{C}$ with $\V(x_1-0.1)$ gives two points that limited to the point $(1,0,0)$.
However there are no real points in $\widehat{C} \cap \V(x_1+0.1)$ that limit to $(1,0,0)$, 
so $C_1^{-}$ is empty. Indeed, tracking the point $p\approx(0.8293, 0.1, 0.0354)\in \Lambda_1^{\tau}$ along the path $\widehat{C}\cap \V(x_1-0.1 e^{i\theta})$
for $\theta\in [0,\pi]$ yields a complex point in $\widehat{C}\cap  \V(x_1+0.1 )$. 
As in Example~\ref{ex:IllusEx}, applying Algorithm~2 to the point $p$ 
gives ray $r = (0,2,3)$.  Every coordinate of $p$ is positive, so we add the ray $(r_0-r_1,r_0-r_2) = (-2,-3)$ to $\TropR(I)$ with sign vector $(+,+)$. 
The other point point $p\approx(0.9635, 0.1, -0.0317)\in \Lambda_1^{\tau}$ gives  $(-2,-3)\in \TropR(I)$ with sign vector $(+,-)$. Repeating this process for the points in $\Lambda_0^{\tau}$,
gives the signed real tropical variety $\TropR(C)$  in Figure~\ref{fig:patchesReal}. 
\end{Example}

\begin{figure}[!ht]
\begin{center}
\includegraphics[height=1.8in]{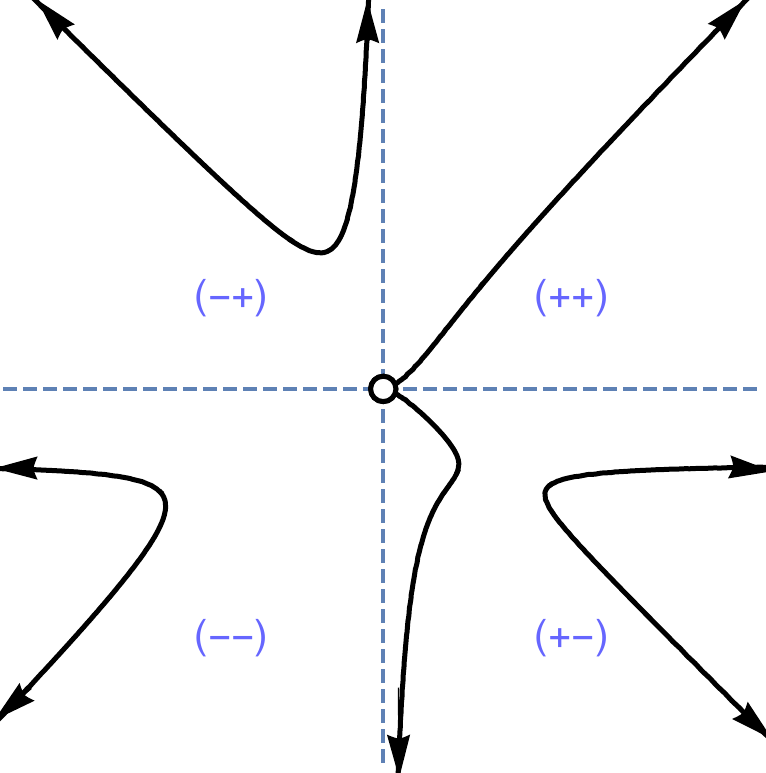} \quad \quad 
\includegraphics[height=1.8in]{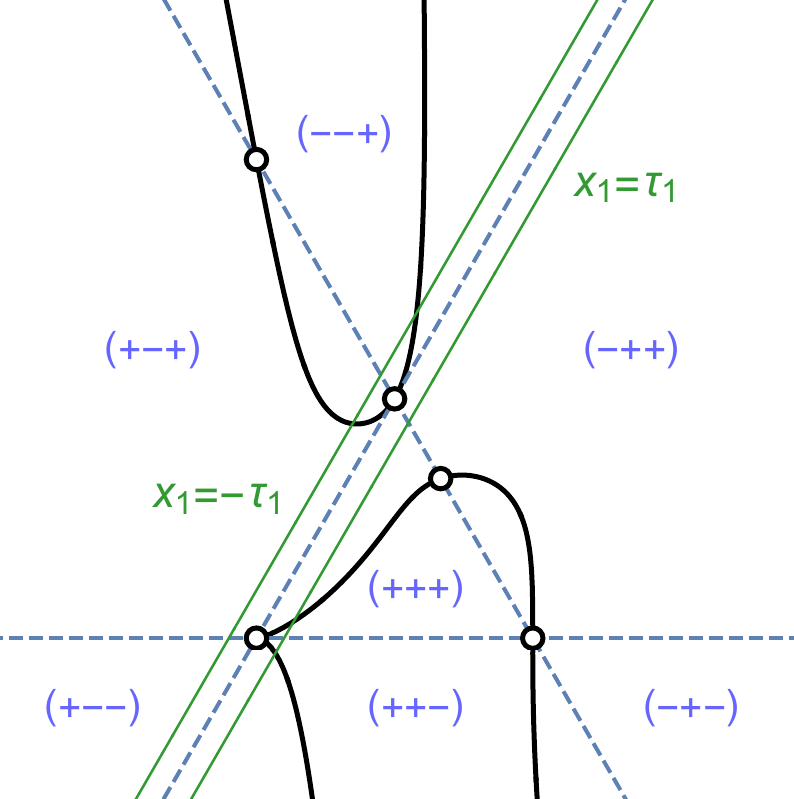}\quad \quad 
\includegraphics[height=1.8in]{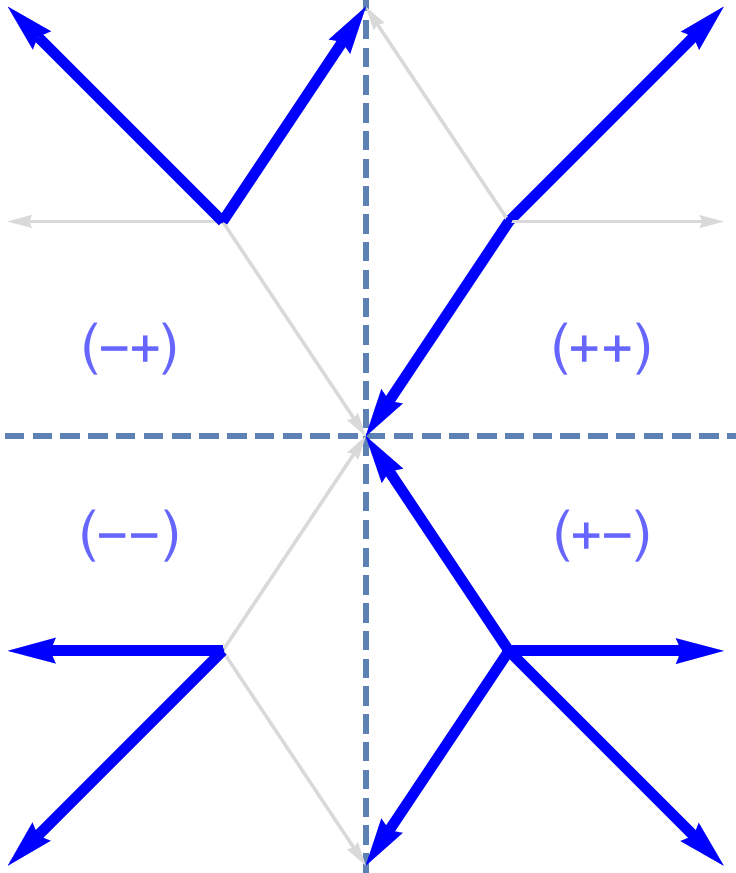}
\end{center}
\caption{Affine curves $C$ and $\widehat{C}$ of the real quartic in Example~\ref{ex:IllusExReal} and the signed real tropical variety $\TropR(C)$ 
contained in reflected copies of $\TropC(C)$.}\label{fig:patchesReal}
\end{figure}

\begin{Thm}
Algorithm~\ref{alg:TropR} computes the primitive vectors with signs of $\TropR(C)$.
\end{Thm}
\begin{proof}
Let $\mathcal{T}$ be the set of vectors with signs computed by Algorithm~\ref{alg:TropR}.
Suppose that $u$ is a primitive vector of $\TropR(C)$ with sign $\sigma \in \{\pm1\}^n$.
By Proposition~\ref{prop:CtoCHat}, we have that 
$w = (0,u) - \max\{0, u_1,\hdots,u_n\}\1$ is a primitive
vector of $\TropR(\widehat{C})$.  The proof of Proposition~\ref{prop:CtoCHat} shows that
we can take $w$ with sign $(\lambda, \lambda \sigma)$ for some $\lambda\in \{\pm 1\}$. 
Let $j$ be the first index for which $w_j<0$ and let $\tau_j$ be the output of Algorithm~\ref{alg:EOZ} with $X=\widehat{C}$. 

Let $K = \R\ps_{conv}$.  Then there is some point $y\in V_K(\widehat{I} + \langle x_j\pm \tau_j t\rangle )$ with valuation $-|1/w_j|w$
and sign $(\lambda, \lambda \sigma)$.  In particular, $y_j = \lambda \sigma_j \cdot \tau_j t$.
By Proposition~\ref{Prop:Puiseux2Path}, $y$ converges for $t\in [0,1]$ 
and $p=y(1)$ and $\tau =  \lambda \sigma_j \cdot \tau_j$  defines a viable input for Algorithm~\ref{alg:trop_rays}.
As in the proof of Theorem~\ref{thm:TropC}, $p$ belongs to $C^{+}_j$ and $\Lambda^{+}_j$ if  $\lambda \sigma_j=1$, and 
$p$ belongs to $C^{-}_j$ and $\Lambda^{-}_j$ if  $\lambda \sigma_j=1$. 
By Proposition~\ref{Prop:Path2Puiseux}, Algorithm~\ref{alg:trop_rays} returns primitive vector $r$, where 
$(1/r_j)r = -|1/w_j |w$ is the valuation of $y$.  Thus the vector $u = (r_0-r_1, \hdots, r_0-r_n)$ is added to~$\mathcal{T}$ with
sign vector ${\rm sign}(p_0p_1, \hdots, p_0p_n)$.  To see that ${\rm sign}(p_0p_1, \hdots, p_0p_n) ={\rm sign}(y_0y_1, \hdots, y_0y_n) =\sigma$, 
note that for all $s\in (0,1]$ and any $k =0,\hdots, n$, the coordinate function $y_k(s)$ is non-zero and therefore has constant sign. 
If not, then $y_k(s) = 0$ for some $0<s\leq1$, then $y(s)\in\Lambda$ and $\tau_j s= |y_j(s)| \in \abs(\pi_j(\Lambda))$, which contradicts 
our choice of $\tau_j<  \abs(\pi_j(\Lambda))$
computed by Algorithm~\ref{alg:EOZ}.  Therefore $\TropR(C)$ is contained in $\mathcal{T}$. 
By Proposition~\ref{Prop:Path2Puiseux}, every path used in Algorithm~\ref{alg:trop_rays} 
comes from some tuple of Puiseux series yielding equality. 
\end{proof}

%
%

\section{Implementation}
\label{sec:implementation}

We have implemented the algorithms from Section~\ref{sec:Algorithms} using
a combination of Matlab and {\tt Bertini} \cite{BHSW06},
which is available at \cite{BertiniTropical}.
This section briefly summarizes the key steps with respect to implementing
the algorithms in Section~\ref{sec:Algorithms}.

\subsection{Witness sets}

The input, namely an irreducible curve $C\subset(\C^*)^n$, 
is represented by a {\em witness set} (see, e.g.,
\cite[Chap.~13]{SW05} for more details).
A witness set for the curve $C$ is a triple $\{f,\mathcal{L},W\}$ where:
\begin{enumerate}
\item $f$ is a polynomial system such that $C$ is an irreducible component of the zero set of $f$,
\item $\mathcal{H}$ is a general hyperplane in $\C^n$, and
\item $W = C\cap \mathcal{H}$.
\end{enumerate}
In this context, a general hyperplane $\mathcal{H}$ 
is a hyperplane that intersects $C$ transversely, 
i.e.,~\mbox{$C\cap \mathcal{H}\subset(\C^*)^n$}
consists of $\deg(C)$ distinct points.

The use of witness sets allows for the restriction of computations
to the irreducible curve inside the zero set of $f$ of interest.  
Moreover, one can easily produce a witness set for $\widehat{C}\subset\C^{n+1}$
as described in Section~\ref{sec:patching} from a witness set for $C$.

If $C$ has multiplicity $> 1$ with respect to $f$, then one can 
utilize deflation techniques for positive-dimensional
components which do not add auxiliary variables, \mbox{e.g.,~\cite{Deflation,Isosing}}, 
to reduce to the multiplicity $1$ case.

\subsection{Implementation of Algorithm~\ref{alg:EOZ}}
\label{sec:impl_EOZ}

The key to Algorithm~\ref{alg:EOZ} is in the computation of the set $S$
consisting of critical points.  Following the notation of Section~\ref{sec:EOZ},
we first compute the set $R\subset X\times\P^{n-1}$
which solves
\begin{equation}
\begin{bmatrix}
f(x) \\ Jf(x)_{\widehat{j}}\cdot \nu \end{bmatrix}
=0 \label{eqn:detjac_impl}
\end{equation}
Recall that $Jf(x)_{\widehat{j}}$ denotes the Jacobian matrix of $f$
evaluated at $x$ with the $j^{\rm th}$ column removed.
In particular, since $X$ has multiplicity $1$ with respect to $f$, 
the set $S = \pi(R)$ is a finite set where $\pi(x,\nu) = x$.

By starting with a witness set for $X$, 
regeneration \cite{hauenstein2011regeneration,RegenExtension}
can be used to compute $R$.  Moreover,
the returned value of Algorithm~\ref{alg:EOZ}
in our implementation is $\tau_j = \min(T_j^\ast)/2$ when $T_j^\ast \neq \emptyset$ and
$\tau_j = 1/10$ otherwise.

\subsection{Implementation of Algorithm~\ref{alg:trop_rays}}
\label{sec:impl_val_mult}

The two key steps in Algorithm~\ref{alg:trop_rays} 
are computing the cycle number and the Cauchy integral.

Once inside the endgame operating zone, the cycle number the number of loops around $x_j=0$ (parametrized by $x_j = \tau e^{i\theta}$)
necessary to return to the starting point $p$.  The cycle number is therefore at most the degree of curve.
By using loops based on regular polygons,
this results in tracking using a sequence of so-called 
Newton homotopies, in which only the 
value of $x_j$ depends explicitly on the path tracking parameter.  
These computations can be performed certifiably \cite{CertifiedTrack1,CertifiedTrack2}.
By using regular polygons, the data from this computation
is reused for the computation of the Cauchy integral which is described next.

By uniformizing via the cycle number, we reduce to power series computations
and the coefficients are computed using Cauchy integrals
where the integrand is periodic with period~$2\pi$.  
Hence, the trapezoidal rule, which is computed using
the data from the regular polygonal paths described above, 
is exponentially convergent \cite{trefethen2014exponentially}.
Due to the exponential convergence, the challenge of deciding
if the integral is zero or nonzero is greatly reduced.  
That is, the exponential convergence allows one to validate this decision by recomputing using 
the trapezoid rule with more sample points computed more accurately using
higher~precision~arithmetic.

\subsection{Implementation of Algorithm~\ref{alg:TropC}}

The computations performed in Algorithm~\ref{alg:TropC} reduce
to tracking solution paths as the hyperplane $\mathcal{H}$
in the witness set for $\widehat{C}$ is deformed.  
To calculate $\widehat{C}\cap\V(x_0\cdots x_n)$, 
we actually compute $\widehat{C}\cap\V(x_j)$ for $j=0,\hdots, n$ and take their union. 
Each of these is obtained by tracking the solution paths defined
by $\widehat{C}\cap(t\cdot\mathcal{H} + (1-t)\cdot\V(x_j))$
from $t = 1$ to $t = 0$.  The computations in Step 10 and Step 11 
follow similarly.

\subsection{Implementation of Algorithm~\ref{alg:TropR}}

One key difference between Algorithm~\ref{alg:TropC} and
Algorithm~\ref{alg:TropR} is that 
only {\em real} points are retained in Algorithm~\ref{alg:TropR}.
In our implementation, the determination of reality is based
on a user-determined numerical threshold.  By recomputing
the points more accurately, the imaginary parts should
limit to $0$ at a commensurate rate.   One could also 
certify reality by using \cite{hauenstein2012algorithm}.

\subsection{Incomplete intersections}

When the polynomial system $f$ in a witness set for the curve $C\subset(\C^*)^n$
consists of more than $n-1$ polynomials, the standard approach in numerical
algebraic geometry is to replace~$f$ by a randomization of the form $A\cdot f$.
For example, the twisted cubic curve $C\subset\C^3$ is defined by
the polynomial system $f=\{y-x^2, z-x y, y^2-xz\}$. 
The zeros set of the sufficiently general 
randomization, say $g = \{y-x^2 + 2(y^2 - xz), z-xy + 3(y^2-xz)\}$,
consists of $C$ and a line.  In particular, one virtue of 
working with witness sets is that one can replace $f$ by $g$ in a
witness set for $C$.  Our implementation relies upon the user
to provide a randomization. 

\subsection{Computational challenges}

In our experience, the majority of the computational time in computing
$\TropC$ and $\TropR$ using Algorithm~\ref{alg:TropC} and Algorithm~\ref{alg:TropR}, respectively, is in the computation of $\tau_j$ via Algorithm~\ref{alg:EOZ}.

Another issue is that endpoints that lie on a coordinate axis are frequently singular, as noted in \cite{jensen2014computing}.  By using endgame methods, such as 
Cauchy's endgame \cite{Cauchy},
with adaptive precision computations \cite{AMP}, 
one is able to accurately compute such endpoints.  This
can be computationally expensive due to the numerical ill-conditioning
near singular solutions.

%
%

\section{Non-planar examples}\label{sec:examples}

In this section, we compute the tropicalizations of real and complex curves in more than two dimensions. 
In our first example, we replicate the main example from \cite{jensen2014computing} and extend it to the real numbers.  
Our second example is the central curve of a linear program that formed part of the recent counterexample 
to the continuous Hirsch conjecture \cite{TropCentralPath}. 

\subsection{$A$-polynomial of a knot}\label{ex:knot}

First, we consider the real and complex tropical varieties of a curve 
whose image under a monomial map is the 
plane curve defined by the $A$-polynomial for the knot $8_1$. 
This curve is the main example of \cite{jensen2014computing} and 
is a component of the reducible variety defined by the ideal 
\begin{equation}
\begin{tabular}{c}
$I = \langle \ z_1+w_1-1, \ z_2+w_2-1, \ z_3+w_3-1, \ z_4+w_4-1, \ z_5+w_5-1,$\hspace{.8in} \smallskip \\
$~~~~~~w_2 w_4-z_2 z_4 w_1 w_5, \
z_2 z_4 z_5^2 w_1^2 - z_1^2 w_2 w_3 w_4 w_5, \
w_3^2-z_3^2 w_1, \
w_5^2-z_2 z_4 z_5^2\ \rangle$.
\end{tabular}
\label{eqn:knotsystem}
\end{equation}
The set $\V_{\C^*}(I)$ consists of a degree $22$ curve $C$ of interest.
We note that $\V_{\C}(I)$ also contains~$4$ additional curves, each
having degree $3$ and multiplicity greater than one.
Algorithms~\ref{alg:TropC} and~\ref{alg:TropR} compute the complex and real tropical varieties of $C$, consisting of eight and seven rays, respectively, as summarized below: \bigskip 

\noindent\begin{tabular} {c c rrrrrrrrrrrr}
\textnormal{complex } &\textnormal{real contribution to }  & \multicolumn{11}{c}{\ }  \\
\textnormal{multiplicity} &\textnormal{complex multiplicity} & \multicolumn{11}{c}{\hspace{.6in}\textnormal{primitive element of ray in $\Trop(I)$}} \smallskip \\
3 \!&\! 3 \!&\!$(\hphantom{-} 0, $\!&\!\!$ 1, $\!&\!$ 0, $\!&\!$ -1, $\!&\!$ 1, $\!&\!$ 0, $\!&\!$ 1, $\!&\!$ 0, $\!&\!$ 0, $\!&\!$ 1)$ \\
4 \!&\! 2 \!&\!$(-1, $\!&\!$ 1, $\!&\!$ 0, $\!&\!$ 1, $\!&\!$ -1, $\!&\!$ 0, $\!&\!$ 1, $\!&\!$ 0, $\!&\!$ 1, $\!&\!$ 0 )$ \\
3 \!&\! 3 \!&\!$(\hphantom{-}0, $\!&\!$ -1, $\!&\!$ 0, $\!&\!$ 1, $\!&\!$ 1, $\!&\!$ 0, $\!&\!$ 0, $\!&\!$ 0, $\!&\!$ 1, $\!&\!$ 1)$ \\
1 \!&\! 1 \!&\!$(\hphantom{-}0, $\!&\!$ 0, $\!&\!$ 0, $\!&\!$ -2, $\!&\!$ 0, $\!&\!$ -4, $\!&\!$ -7, $\!&\!$ -2, $\!&\!$ 0, $\!&\!$ -1)$\\
1 \!&\! 1 \!&\!$(\hphantom{-}0, $\!&\!$ -2, $\!&\!$ 0, $\!&\!$ 0, $\!&\!$ 0, $\!&\!$ -4, $\!&\!$ 0, $\!&\!$ -2, $\!&\!$ -7, $\!&\!$ -1)$\\
2 \!&\! 2 \!&\!$(\hphantom{-}2, $\!&\!$ -2, $\!&\!$ -1, $\!&\!$ 0, $\!&\!$ 0, $\!&\!$ 2, $\!&\!$ 0, $\!&\!$ 0, $\!&\!$ -1, $\!&\!$ -1)$ \\
2 \!&\! 2 \!&\!$(\hphantom{-}2, $\!&\!$ 0, $\!&\!$ -1, $\!&\!$ -2, $\!&\!$ 0, $\!&\!$ 2, $\!&\!$ -1, $\!&\!$ 0, $\!&\!$ 0, $\!&\!$ -1)$  \\
2 \!&\! 0 \!&\!$(-2, $\!&\!$ 1, $\!&\!$ 2, $\!&\!$ 1, $\!&\!$ -1, $\!&\!$ 0, $\!&\!$ 1, $\!&\!$ 2, $\!&\!$ 1, $\!&\!$ 0)$ 
\end{tabular} \bigskip 

The complex tropical variety replicates Example 4.1 in \cite{jensen2014computing}. 
Of the fourteen intersection points of the corresponding curve $\widehat{C}$ with the union of the 
coordinate hyperplanes, twelve are real.  
Two of these points have the form $(h,z,w) = (0, 0, -1, 0, -a, 0, 0, 1, 0, a, 0)$, where $h$ is the added homogenzing variable and $a$ satisfies
$a^4 - 2a^3 - 5a ^2 - 2a + 1=0$.
Both of these real points contribute one towards the 
multiplicity of the ray $\val(h,z,w) = (1,2,0,1,0,2,1,0,1,0,1)$, which 
corresponds to the second ray listed in the table above.
The two complex values of $a$ solving the quartic equation above 
also contribute one each to the complex multiplicity of this ray, 
giving a total complex multiplicity of four.

\subsection{The central curve of a linear program} \label{ex:central_curve}

In a recent series of papers, Allamigeon, Benchimol, Gaubert, and  Joswig \cite{TropCentralPath,TropSimplex} 
develop a theory of tropical linear programming.  Among other things, this enabled them to produce a counterexample to 
the continuous Hirsch conjecture regarding 
the total curvature of the central path of a linear program.

The central path of a linear program is a segment of an algebraic curve, called the central curve, which 
joins the analytic center of the feasible polytope and its optimal vertex \cite{BayerLagarias, CentralCurveDSV}.
A family of linear programs was presented in \cite{ TropCentralPath}
whose central paths have total curvature that grows exponentially. 
We compute the real tropical curve of the central curve of one member of this family, 
specifically, the linear program from \cite[\S 4]{TropCentralPath} with $r=1, \ t=4$:
\begin{align*}
\text{minimize } \ \  v_0 \  \ \text{ subject to} \ \ & -u_0 + t \geq 0, \ \ -v_0 + t^2  \geq 0, \ \ u_1 \geq 0, \ \ v_1 \geq 0,\\
& t^{1/2} (u_0+v_0) - v_1\geq 0, \ \
t u_0 - u_1 \geq 0, \ \
t v_0 - u_1 \geq 0. 
\end{align*}
The central curve of this linear program is the projection of a curve in $\R^{18}$ with 14 auxiliary variables $(x,s) = (x_1, \hdots, x_7, s_1, \hdots, s_7)$
defined by the ideal 
\begin{align} \small
I = & \langle x_1 s_1 - x_2 s_2, \ x_1 s_1 - x_3 s_3,\ x_1 s_1 - x_4 s_4, \ x_1 s_1 - x_5 s_5,\  x_1 s_1 - x_6 s_6,\ x_1 s_1 - x_7 s_7, \nonumber \\
& -u_0 + t  - x_1, \  -v_0 + t^2  - x_2, \ u_1 - x_3,\ v_1 - x_4,\ t^{1/2} (u_0+v_0) - v_1 - x_5, \  t u_0 - u_1 - x_6, \nonumber\\
& \, t v_0 - u_1 - x_7,\ s_1 - t^{1/2} s_5 - t s_6,\ s_2 - t^{1/2} s_5 - t s_7-1,\ s_3 -  s_6 - s_7, \ s_4 - s_5\rangle .  \label{eqn:central_curve}
\end{align}

The system consists of 18 variables and 17 equations.  The algebraic variety of this ideal consists of two linear $3$-spaces,
five planes, four lines, and a degree 10 curve, say $C$.  
The linear components belong to the \emph{disjoint support variety} 
described in \cite[\S 7]{CentralCurveDSV}. 

Let $\widehat{C}$ be the curve corresponding to $C$ as in Section~\ref{sec:patching}.
The total number of points contained in the intersection of $\widehat{C}$ 
with the union of the coordinate hyperplanes is 22, all of which are real.  
There is a distinct path through each of these 22 points 
with cycle number one and 
each point gives a single contribution to the multiplicity of a ray in 
the complex tropical curve.
In particular, the real and complex tropical curves are equal:

\begin{figure}
\begin{center}
\raisebox{-0.5\height}{\includegraphics[width = 2.9in, clip=true, viewport=20 200 500 600]{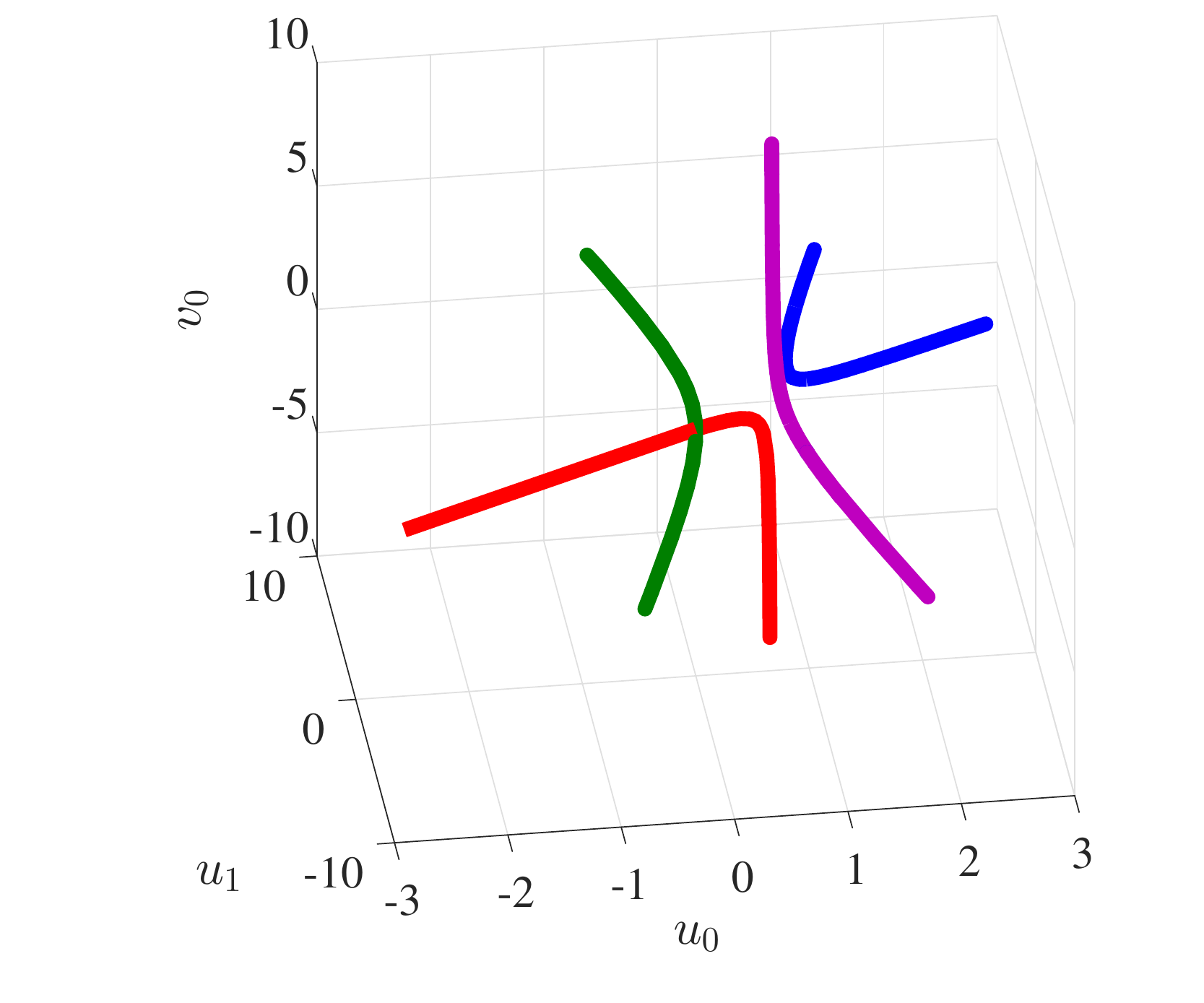}}
\raisebox{-0.5\height}{\includegraphics[width = 3in, clip=true, viewport=0 180 550 600]{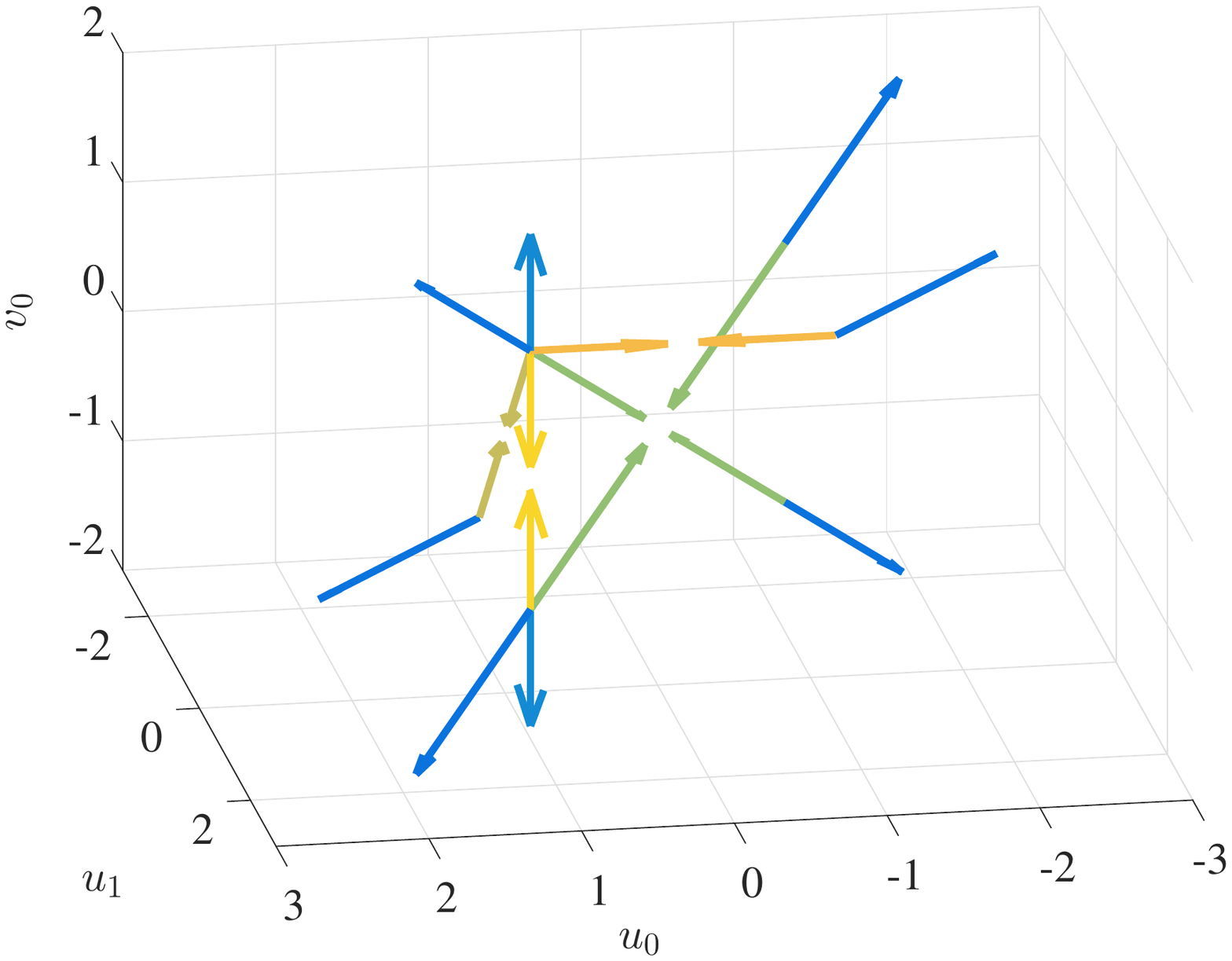}}
\caption{Left: The projection of the degree 10 central curve defined by \eqref{eqn:central_curve} in Example~\ref{ex:central_curve} onto the coordinates $(u_0, u_1, v_0)$.  
Right: Corresponding signed real tropical variety. }
\label{fig:central_curve}
\end{center}
\end{figure}

\begin{center}
\begin{tabular} {crrrrrrrrrrrrrrrrrr}
\textnormal{multiplicity} & \multicolumn{18}{c}{\textnormal{primitive element of ray in $\Trop(I)$ in $(x,s,u_0,v_0,u_1,v_1)$}}\smallskip \\
6 &$ (\hphantom{-}0,$\!\!\!&\!\!\!$  0,$\!\!\!&\!\!\!$  0,$\!\!\!&\!\!\!$  0,$\!\!\!&\!\!\!$  0,$\!\!\!&\!\!\!$  0,$\!\!\!&\!\!\!$  0,$\!\!\!&\!\!\!$  1,$\!\!\!&\!\!\!$  1,$\!\!\!&\!\!\!$  1,$\!\!\!&\!\!\!$  1,$\!\!\!&\!\!\!$  1,$\!\!\!&\!\!\!$  1,$\!\!\!&\!\!\!$  1,$\!\!\!&\!\!\!$  0,$\!\!\!&\!\!\!$  0,$\!\!\!&\!\!\!$  0,$\!\!\!&\!\!\!$  0)  $\\ 
3 &$  (\hphantom{-}1,$\!\!\!&\!\!\!$  1,$\!\!\!&\!\!\!$  1,$\!\!\!&\!\!\!$  1,$\!\!\!&\!\!\!$  1,$\!\!\!&\!\!\!$  1,$\!\!\!&\!\!\!$  1,$\!\!\!&\!\!\!$  0,$\!\!\!&\!\!\!$  0,$\!\!\!&\!\!\!$  0,$\!\!\!&\!\!\!$  0,$\!\!\!&\!\!\!$  0,$\!\!\!&\!\!\!$  0,$\!\!\!&\!\!\!$  0,$\!\!\!&\!\!\!$  1,$\!\!\!&\!\!\!$  1,$\!\!\!&\!\!\!$  1,$\!\!\!&\!\!\!$  1)  $\\ 
1 &$ (\hphantom{-}0,$\!\!\!&\!\!\!$  1,$\!\!\!&\!\!\!$  0,$\!\!\!&\!\!\!$  1,$\!\!\!&\!\!\!$  1,$\!\!\!&\!\!\!$  0,$\!\!\!&\!\!\!$  1,$\!\!\!&\!\!\!$  1,$\!\!\!&\!\!\!$  0,$\!\!\!&\!\!\!$  1,$\!\!\!&\!\!\!$  0,$\!\!\!&\!\!\!$  0,$\!\!\!&\!\!\!$  1,$\!\!\!&\!\!\!$  0,$\!\!\!&\!\!\!$  0,$\!\!\!&\!\!\!$  1,$\!\!\!&\!\!\!$  0,$\!\!\!&\!\!\!$  1) $\\ 
1 &$(-1,$\!\!\!&\!\!\!$  0,$\!\!\!&\!\!\!$  0,$\!\!\!&\!\!\!$  0,$\!\!\!&\!\!\!$  0,$\!\!\!&\!\!\!$  -1,$\!\!\!&\!\!\!$  -1,$\!\!\!&\!\!\!$  0,$\!\!\!&\!\!\!$  -1,$\!\!\!&\!\!\!$  -1,$\!\!\!&\!\!\!$  -1,$\!\!\!&\!\!\!$  -1,$\!\!\!&\!\!\!$  0,$\!\!\!&\!\!\!$  0,$\!\!\!&\!\!\!$  0,$\!\!\!&\!\!\!$  0,$\!\!\!&\!\!\!$  0,$\!\!\!&\!\!\!$  0)  $\\ 
2 &$(-1,$\!\!\!&\!\!\!$  0,$\!\!\!&\!\!\!$  0,$\!\!\!&\!\!\!$  -1,$\!\!\!&\!\!\!$  -1,$\!\!\!&\!\!\!$  0,$\!\!\!&\!\!\!$  0,$\!\!\!&\!\!\!$  0,$\!\!\!&\!\!\!$  -1,$\!\!\!&\!\!\!$  -1,$\!\!\!&\!\!\!$  0,$\!\!\!&\!\!\!$  0,$\!\!\!&\!\!\!$  -1,$\!\!\!&\!\!\!$  -1,$\!\!\!&\!\!\!$  0,$\!\!\!&\!\!\!$  0,$\!\!\!&\!\!\!$  0,$\!\!\!&\!\!\!$  -1) $\\ 
4 &$ (\hphantom{-}0,$\!\!\!&\!\!\!$  -1,$\!\!\!&\!\!\!$  0,$\!\!\!&\!\!\!$  0,$\!\!\!&\!\!\!$  0,$\!\!\!&\!\!\!$  0,$\!\!\!&\!\!\!$  0,$\!\!\!&\!\!\!$  -1,$\!\!\!&\!\!\!$  0,$\!\!\!&\!\!\!$  -1,$\!\!\!&\!\!\!$  -1,$\!\!\!&\!\!\!$  -1,$\!\!\!&\!\!\!$  -1,$\!\!\!&\!\!\!$  -1,$\!\!\!&\!\!\!$  0,$\!\!\!&\!\!\!$  0,$\!\!\!&\!\!\!$  0,$\!\!\!&\!\!\!$  0) $\\ 
2 &$ (\hphantom{-}0,$\!\!\!&\!\!\!$  0,$\!\!\!&\!\!\!$  -1,$\!\!\!&\!\!\!$  -1,$\!\!\!&\!\!\!$  -1,$\!\!\!&\!\!\!$  -1,$\!\!\!&\!\!\!$  -1,$\!\!\!&\!\!\!$  -1,$\!\!\!&\!\!\!$  -1,$\!\!\!&\!\!\!$  0,$\!\!\!&\!\!\!$  0,$\!\!\!&\!\!\!$  0,$\!\!\!&\!\!\!$  0,$\!\!\!&\!\!\!$  0,$\!\!\!&\!\!\!$  -1,$\!\!\!&\!\!\!$  -1,$\!\!\!&\!\!\!$  -1,$\!\!\!&\!\!\!$  -1) $\\ 
1 &$(\hphantom{-}0,$\!\!\!&\!\!\!$  0,$\!\!\!&\!\!\!$  -1,$\!\!\!&\!\!\!$  0,$\!\!\!&\!\!\!$  0,$\!\!\!&\!\!\!$  0,$\!\!\!&\!\!\!$  -1,$\!\!\!&\!\!\!$  -1,$\!\!\!&\!\!\!$  -1,$\!\!\!&\!\!\!$  0,$\!\!\!&\!\!\!$  -1,$\!\!\!&\!\!\!$  -1,$\!\!\!&\!\!\!$  -1,$\!\!\!&\!\!\!$  0,$\!\!\!&\!\!\!$  0,$\!\!\!&\!\!\!$  -1,$\!\!\!&\!\!\!$  -1,$\!\!\!&\!\!\!$  0)  $\\ 
1 &$ (\hphantom{-}0,$\!\!\!&\!\!\!$  0,$\!\!\!&\!\!\!$  0,$\!\!\!&\!\!\!$  0,$\!\!\!&\!\!\!$  0,$\!\!\!&\!\!\!$  0,$\!\!\!&\!\!\!$  0,$\!\!\!&\!\!\!$  0,$\!\!\!&\!\!\!$  0,$\!\!\!&\!\!\!$  0,$\!\!\!&\!\!\!$  0,$\!\!\!&\!\!\!$  0,$\!\!\!&\!\!\!$  0,$\!\!\!&\!\!\!$  0,$\!\!\!&\!\!\!$  -1,$\!\!\!&\!\!\!$  0,$\!\!\!&\!\!\!$  0,$\!\!\!&\!\!\!$  0)$\\ 
1 &$(\hphantom{-}0,$\!\!\!&\!\!\!$  0,$\!\!\!&\!\!\!$  0,$\!\!\!&\!\!\!$  0,$\!\!\!&\!\!\!$  0,$\!\!\!&\!\!\!$  0,$\!\!\!&\!\!\!$  0,$\!\!\!&\!\!\!$  0,$\!\!\!&\!\!\!$  0,$\!\!\!&\!\!\!$  0,$\!\!\!&\!\!\!$  0,$\!\!\!&\!\!\!$  0,$\!\!\!&\!\!\!$  0,$\!\!\!&\!\!\!$  0,$\!\!\!&\!\!\!$  0,$\!\!\!&\!\!\!$  -1,$\!\!\!&\!\!\!$  0,$\!\!\!&\!\!\!$  0)$\\ \end{tabular}
\end{center}

Figure~\ref{fig:central_curve} shows a projection of the real points in $C$
and the corresponding signed real tropical curve onto coordinates $(u_0, u_1, v_0)$.  The colors of the real curve indicate the image of connected components.  There is an intersection between the green and red segments, but the magenta and blue segments do not intersect.

Generally, the signed real tropical variety of the central curve of a linear program should be closely related to the oriented matroid associated 
with the input data \cite{bachem1992linear}. 

\section{Conclusion}

This paper and its predecessors \cite{HauensteinSottile12, jensen2014computing}
demonstrate the potential for using the tools of numerical algebraic geometry to compute tropical varieties.  The recent expansion of numerical tools for real varieties  
makes numerical algebraic geometric methods 
particular effective for computing real tropical varieties.  
The next natural step for the development of these techniques is to compute tropicalizations of real and complex surfaces.  

One important motivation for computing tropical surfaces is that it
would enable the computation of tropical curves defined by 
polynomials with non-constant coefficients, which are of interest for many of the applications mentioned earlier, e.g., \cite{TropCentralPath, ViroPatchworking}. 
Real tropical surfaces also provide significantly more subtleties than curves. 
For example, the real tropical variety of a surface may not be pure-dimensional 
and may not be a sub-fan of the complex tropical variety.  
With recent developments related to numerically decomposing real surfaces in any 
dimension, e.g., \cite{BertiniReal}, we are optimistic that a
few new ideas building on the ability to compute 
real tropical curves will be enough to compute real tropical surfaces. 

\section*{Acknowledgments}\label{acknowledge}

DAB and JDH were partially 
supported by Sloan Research Fellowship and NSF ACI-1460032.
CV~was partially supported by NSF grant DMS-1204447 and the FRPD program at North Carolina State University.
We would like to thank Anton Leykin for helpful comments.

\bibliographystyle{plain}


\begin{thebibliography}{10}

\bibitem{Alessandrini13}
D. Alessandrini.
\newblock Logarithmic limit sets of real semi-algebraic sets.
\newblock {\em Adv. Geom.}, 13(1):155--190, 2013.

\bibitem{TropCentralPath}
X. Allamigeon, P. Benchimol, S. Gaubert, and M. Joswig.
\newblock Long and winding central paths.
\newblock \url{arXiv:1405.4161}, 2014.

\bibitem{TropSimplex}
X. Allamigeon, P. Benchimol, S. Gaubert, and M. Joswig.
\newblock Tropicalizing the simplex algorithm.
\newblock {\em SIAM J. Discrete Math.}, 29(2):751--795, 2015.

\bibitem{PositiveBergman}
F. Ardila, C. Klivans, and L. Williams.
\newblock The positive {B}ergman complex of an oriented matroid.
\newblock {\em European J. Combin.}, 27(4):577--591, 2006.

\bibitem{bachem1992linear}
A.~Bachem and W.~Kern.
\newblock {\em Linear Programming Duality: An Introduction to Oriented
  Matroids}.
\newblock Universitext. Springer Berlin Heidelberg, 1992.

\bibitem{AMP}
D.J. Bates, J.D. Hauenstein, A.J. Sommese, and C.W. Wampler.
\newblock Adaptive multiprecision path tracking.
\newblock {\em SIAM J. Numer. Anal.}, 46(2):722--746, 2008.

\bibitem{BHSW06}
D.J. Bates, J.D. Hauenstein, A.J. Sommese, and C.W. Wampler.
\newblock Bertini: Software for numerical algebraic geometry.
\newblock Available at \url{bertini.nd.edu}.

\bibitem{BayerLagarias}
D.A. Bayer and J.C. Lagarias.
\newblock The nonlinear geometry of linear programming. {II}. {L}egendre
  transform coordinates and central trajectories.
\newblock {\em Trans. Amer. Math. Soc.}, 314(2):527--581, 1989.

\bibitem{BertiniReal}
D.A. Brake, D.J. Bates, W. Hao, J.D. Hauenstein, A.J. Sommese, and C.W. Wampler.
\newblock {${\rm Bertini}_-{\rm real}$}: software for one- and two-dimensional
  real algebraic sets.
\newblock In {\em Mathematical software---{ICMS} 2014}, volume 8592 of {\em
  Lecture Notes in Comput. Sci.}, pages 175--182. Springer, Heidelberg, 2014.

\bibitem{BertiniTropical}
D.A. Brake, J.D. Hauenstein, and C. Vinzant.
\newblock {${\rm Bertini}_-{\rm tropical}$}.
\newblock Available at \url{bertini.nd.edu/tropical}.

\bibitem{CentralCurveDSV}
J.A. De~Loera, B. Sturmfels, and C. Vinzant.
\newblock The central curve in linear programming.
\newblock {\em Found. Comput. Math.}, 12(4):509--540, 2012.

\bibitem{CertifiedTrack1}
J.D. Hauenstein, I. Haywood, and A.C. Liddell, Jr.
\newblock An {\em a posteriori} certification algorithm for Newton homotopies.
\newblock In {\em Proceedings of the 39th International Symposium on Symbolic and Algebraic Computation}, ACM, New York, 2014, pp. 248--255. 

\bibitem{CertifiedTrack2}
J.D. Hauenstein and A.C. Liddell, Jr.
\newblock Certified predictor-corrector tracking for Newton homotopies.
\newblock {\em Journal of Symbolic Computation}, 74:239--254, 2016.

\bibitem{Deflation}
J.D. Hauenstein, B. Mourrain, and A. Szanto.
\newblock Certifying isolated singular points and their multiplicity structure.
\newblock In {\em Proceedings of the 40th International Symposium on Symbolic and Algebraic Computation}, ACM, New York, 2015, pp. 213--220.

\bibitem{hauenstein2011regeneration}
J.D. Hauenstein, A.J. Sommese, and C.W. Wampler.
\newblock Regeneration homotopies for solving systems of polynomials.
\newblock {\em Math. Comp.}, 80(273):345--377, 2011.

\bibitem{HauensteinSottile12}
J.D. Hauenstein and F. Sottile.
\newblock Newton polytopes and witness sets.
\newblock {\em Math. Comp. Sci.}, 8(2):235--251, 2014.

\bibitem{hauenstein2012algorithm}
J.D. Hauenstein and F. Sottile.
\newblock Algorithm 921: alpha{C}ertified: certifying solutions to polynomial systems.
\newblock {\em ACM Trans. Math. Softw.}, 38(4):28, 2012.

\bibitem{Isosing}
J.D. Hauenstein and C.W. Wampler.
\newblock Isosingular sets and deflation.
\newblock {\em Found. Comput. Math.}, 13(3):371--403, 2013. 

\bibitem{RegenExtension}
J.D. Hauenstein and C.W. Wampler.
\newblock Unification and extension of intersection algorithms in numerical algebraic geometry.
\newblock Preprint, available at \url{www.nd.edu/~jhauenst/preprints}.

\bibitem{jensen2014computing}
A.~ Jensen, A.~Leykin, and J.~Yu.
\newblock Computing tropical curves via homotopy continuation.
\newblock {\em Exp. Math.}, 25(1):83--93, 2016.

\bibitem{MSbook}
D. Maclagan and B. Sturmfels.
\newblock {\em Introduction to tropical geometry}, volume 161 of {\em Graduate
  Studies in Mathematics}.
\newblock American Mathematical Society, Providence, RI, 2015.

\bibitem{Mikhalkin}
G. Mikhalkin.
\newblock Enumerative tropical algebraic geometry in {$\Bbb R^2$}.
\newblock {\em J. Amer. Math. Soc.}, 18(2):313--377, 2005.

\bibitem{Cauchy}
A.P. Morgan, A.J. Sommese, and C.W. Wampler.
\newblock Computing singular solutions to nonlinear analytic systems.
\newblock {\em Numer. Math.}, 58(7):669--684, 1991.

\bibitem{SW05}
A.J. Sommese and C.W. Wampler, II.
\newblock {\em The Numerical Solution of Systems of Polynomials Arising in
  Engineering and Science}.
\newblock World Scientific Publishing Co. Pte. Ltd., Hackensack, NJ, 2005.

\bibitem{SpeyerWilliams05}
D. Speyer and L. Williams.
\newblock The tropical totally positive {G}rassmannian.
\newblock {\em J. Algebraic Combin.}, 22(2):189--210, 2005.

\bibitem{Tabera}
L.F. Tabera.
\newblock On real tropical bases and real tropical discriminants.
\newblock {\em Collect. Math.}, 66(1):77--92, 2015.

\bibitem{trefethen2014exponentially}
L.N. Trefethen and J.A.C. Weideman.
\newblock The exponentially convergent trapezoidal rule.
\newblock {\em SIAM Review}, 56(3):385--458, 2014.

\bibitem{Vinzant12}
C. Vinzant.
\newblock Real radical initial ideals.
\newblock {\em J. Algebra}, 352:392--407, 2012.

\bibitem{ViroPatchworking}
O. Viro.
\newblock Real plane algebraic curves: constructions with controlled topology.
\newblock {\em Algebra i Analiz}, 1(5):1--73, 1989.

\bibitem{Viro08}
O. Viro.
\newblock From the sixteenth {H}ilbert problem to tropical geometry.
\newblock {\em Jpn. J. Math.}, 3(2):185--214, 2008.

\end{thebibliography}
\end{document}